\documentclass[11pt]{amsart}
\RequirePackage[OT1]{fontenc}
\usepackage{palatino}
    \usepackage{graphicx}   
    \usepackage{multirow}   
\usepackage{amsthm,amsmath,amsfonts,amssymb,color,stmaryrd,dsfont,natbib,enumitem} 
    \usepackage[linesnumbered,ruled,vlined]{algorithm2e}\SetKwInput{KwInput}{Input}                
\SetKwInput{KwOutput}{Output}              
    \usepackage{bm,bbm}
\usepackage{ulem}
\usepackage{tikz}
\usepackage{hyperref,url}
\usepackage{graphicx} 
\usepackage{stmaryrd}
\usepackage{float} 
\usepackage{subfigure,caption} 

\usepackage[linesnumbered,ruled,vlined]{algorithm2e}\SetKwInput{KwInput}{Input}                
\SetKwInput{KwOutput}{Output}              

\usepackage[a4paper, total={6.7in, 8.5in}]{geometry}
\newcommand{\bgS}{{\mbox{\boldmath$\Sigma$}}}
\newcommand{\bbA}{{\bf A}}
\newcommand{\bba}{{\bf a}}
\newcommand{\bbB}{{\bf B}}
\newcommand{\bbb}{{\bf b}}

\newcommand{\bbd}{{\bf d}}
\newcommand{\bbe}{{\bf e}}
\newcommand{\bbE}{{\bf E}}

\newcommand{\bbF}{{\bf F}}

\newcommand{\bbG}{{\bf G}}
\newcommand{\bbH}{{\bf H}}
\newcommand{\bbh}{{\bf h}}
\newcommand{\bbI}{{\bf I}}

\newcommand{\bbj}{{\bf j}}

\newcommand{\bbK}{{\bf K}}

\newcommand{\bbM}{{\bf M}}

\newcommand{\bbQ}{{\bf Q}}

\newcommand{\bbP}{{\bf P}}

\newcommand{\bbO}{{\bf O}}

\newcommand{\bbr}{{\bf r}}

\newcommand{\bbT}{{\bf T}}

\newcommand{\bbu}{{\bf u}}
\newcommand{\bbV}{{\bf V}}
\newcommand{\bbv}{{\bf v}}

\newcommand{\bbW}{{\bf W}}

\newcommand{\bbX}{{\bf X}}

\newcommand{\bbx}{{\bf x}}

\newcommand{\bbY}{{\bf Y}}
\newcommand{\bby}{{\bf y}}
\newcommand{\bbZ}{{\bf Z}}

\newcommand{\cA}{\bm{{\mathcal A}}}

\newcommand{\bcH}{{\bm{{\mathcal H}}}}

\newcommand{\cK}{{\mathcal K}}
\newcommand{\cM}{{\mathcal M}}

\newcommand{\bcU}{{\bm{{\mathcal U}}}}

\newcommand{\bmth}{\bm\theta}
\newcommand{\bmTh}{\bm\Theta}
\newcommand{\bmPs}{\bm\Psi}

\newcommand{\bmGa}{\bm\Gamma}
\newcommand{\bmga}{\bm\gamma}

\newcommand{\E}{{\mathbb{E}}}

\renewcommand{\P}{{\mathbb{P}}}
\newcommand{\tr}{{\mathrm{tr}}}
\newcommand{\no}{\nonumber}

\newcommand{\asto}{\stackrel{a.s.}{\to}}
\newcommand{\dto}{\stackrel{D}{\to}}

\newcommand{\argmin}{\mathop{\operatorname{arg\,min}}\limits}

\begin{document}

	\theoremstyle{plain}
	\newtheorem{thm}{Theorem}[section]
	\newtheorem{corollary}[thm]{Corollary}
	\newtheorem{defin}[thm]{Definition}
	\newtheorem{prop}[thm]{Proposition}
	\newtheorem{remark}[thm]{Remark}
	\newtheorem{lemma}[thm]{Lemma}
	\newtheorem{assumption}[thm]{Assumption}
%
%
%


\newpage

\begin{center}
\large\bf
KOO Approach for Scalable Variable Selection Problem\\
 in Large-dimensional Regression
\end{center}

\vspace{0.5cm}
\renewcommand{\thefootnote}{\fnsymbol{footnote}}
\hspace{5ex}

\begin{center}
 \begin{minipage}[t]{0.35\textwidth}
\begin{center}
Zhidong Bai \\
\footnotesize {Northeast Normal University, China}\\
{\it baizd@nenu.edu.cn}
\end{center}
\end{minipage}
\hspace{8ex}
\begin{minipage}[t]{0.35\textwidth}
\begin{center}
Kwok Pui Choi \\ 
\footnotesize {National University of Singapore, Singapore}\\
{\it stackp@nus.edu.sg}
\end{center}
\end{minipage}

\end{center}
\vspace{0.2cm}
\begin{center}
 \begin{minipage}[t]{0.35\textwidth}
\begin{center}
Yasunori Fujikoshi \\
\footnotesize {Hiroshima University, Japan}\\
{\it fujikoshi\_y@yahoo.co.jp}
\end{center}
\end{minipage}
\hspace{8ex}
\begin{minipage}[t]{0.35\textwidth}
\begin{center}
Jiang Hu \\ 
\footnotesize {Northeast Normal University, China}\\
{\it huj156@nenu.edu.cn}
\end{center}
\end{minipage}
\end{center}

\vspace{0.5cm}

\begin{center}
 \begin{minipage}{0.8\textwidth}\footnotesize{
 An important issue in many multivariate regression problems is to eliminate candidate predictors with null predictor vectors. In large-dimensional (LD) setting where the numbers of responses and predictors are large, model selection encounters the scalability challenge. Knock-one-out (KOO) statistics hold promise to meet this challenge. In this paper, the almost sure limits and the central limit theorem of the KOO statistics are derived under the LD setting and mild distributional assumptions (finite fourth moments) of the errors. These theoretical results guarantee the strong consistency of a subset selection rule based on the  KOO statistics with a general threshold. For enhancing the robustness of the selection rule, we also propose a bootstrap threshold for the KOO approach. Simulation results support our conclusions and demonstrate the selection probabilities by the KOO approach with the bootstrap threshold outperform the methods using Akaike information threshold, Bayesian information threshold and Mallow's C$_p$ threshold. We compare the proposed KOO approach with those based on information threshold to a chemometrics dataset and a yeast cell-cycle dataset, which suggests our proposed method identifies useful models.}
\end{minipage}
\end{center}

\vspace{0.5cm}
\section{Introduction}
In multivariate statistical analysis, linear regression is a basic and commonly used type of approach. The overall idea of regression is to examine which variables in particular are significant predictors of the outcome variables, and in what way do they indicated by the magnitude and sign of the outcome variables.
Specifically,
\begin{equation}
\label{fullmodel}
	\bbY=\bbX\bm\Theta+\bbE\bgS^{1/2},
\end{equation}
where the $n \times p$ response matrix $\bbY=(y_{ij}) =(\bby_1, \ldots , \bby_n)'$, the $n \times k$ predictor matrix $\bbX= (\tilde{\bbx}_1, \ldots , \tilde{\bbx}_n)'=(\bbx_1,\dots,\bbx_k)$, the $k \times p$ regression coefficient matrix $\bm\Theta=(\bm\theta_1, \dots,\bm\theta_k)'$, the $n\times p$ random errors matrix $\bbE=(\bbe_1,\dots,\bbe_p)=(e_{ij})$ and the $p\times p$ covariance matrix $\bgS$ with full rank.
A main goal in multivariate linear regression (MLR) is to estimate the regression coefficients $\Theta$. The estimates should be such that the
estimated regression plane explains the variation in the
values of the responses with great accuracy.

 Model \eqref{fullmodel} (referred to hereinafter as the full model), however,  is not always satisfactory because some of the predictors may be uncorrelated with the responses. 
We take a simple example to illustrate this fact. Let $\bbj$ be a subset of $[k]=\{1,2,\dots,k\}$, $\bbX_{\bbj}=(\bbx_j, j\in\bbj)$ and ${\bm\Theta}_{\bbj}=(\bm\theta_j, j\in\bbj)'$. Denote model $\bbj$ by
\begin{align}
 \label{eq1}
 M_{\bbj}:\ \ \bbY=\bbX_{\bbj}\bm\Theta_{\bbj}+\bbE\bgS^{1/2}.
\end{align}
The classical linear least-squares solution is to estimate the matrix of regression coefficients $\widehat{\bm\Theta}$ of the full model (\ref{fullmodel}) by
\begin{align*}
	\widehat{\bm\Theta}=(\bbX'\bbX)^{-1}\bbX'\bbY,
\end{align*}
which minimizes the sum of the squares of errors, i.e.,
\begin{align*}
 \widehat{\bm\Theta}= \argmin_{\Theta}\tr(\bbY-\bbX\bmTh)(\bbY-\bbX\bmTh)'.
 \end{align*}
If there exists a predictor vector $\bmth_j=\bf 0$, then the least-squares estimator of the regression coefficients of model $M_{{[k]} \backslash j}$   is
\begin{align*}
	\widehat{\bm\Theta}_{{[k]} \backslash j}=(\bbX_{{[k]} \backslash j}'\bbX_{{[k]} \backslash j})^{-1}\bbX_{{[k]} \backslash j}'
	\bbY,
\end{align*}
It is known that in this case the mean squared error (MSE) of the predictions from $\widehat{\bm\Theta}_{{[k]} \backslash j}$ is smaller than that from $\widehat{\bm\Theta}$ under some mild conditions. Moreover, even though the elements of $\bmth_j$ are not equal to zero but small enough, the MSE of the predictions from  $\widehat{\bm\Theta}_{{[k]} \backslash j}$ is also smaller than that from $\widehat{\bm\Theta}$ (e.g., \cite{FujikoshiU10M}).
Therefore,  removing these ``non-significant" predictors from the full model improves the model. How to determine the significance of each predictor for the response and to select the true model from the full model are important problems in multiple regression model. Here,  the true model is the data-generating model and is denoted by
\begin{align}
 \label{eq1}
 M_{\bbj_*}:\ \ \bbY=\bbX_{\bbj_*}\bm\Theta_{\bbj_*}+\bbE\bgS^{1/2},
\end{align}
where  for all $j\in[k]\backslash \bbj_*$, $\bmth_j=\bf 0$.


To measure the significance of the predictors for the response, one can make use of the regression coefficients, the partial correlation or the multiple correlation coefficient between each predictor and the responses. However, these direct measures are unstable under high-dimensional regression because they all highly depend on the values of each predictor.  Instead, we consider removing one predictor vector from the full model and measuring how much ``information" we lose. Hence, we refer to this kind of statistics KOO (knock-one-out or kick-one-out) statistics in the technical report \citep{BaiF18S}. This KOO idea can be traced back to \cite{NishiiB88S},
 who investigated the discriminant analysis and canonical correlation analysis under fixed dimensions.  In this paper, we study the KOO statistics in high-dimensional responses and predictors.
 The  KOO method was motivated to address the issue of computational complexity in traditional AIC and    BIC methods. Moreover,  we find that the KOO method exhibits excellent stability, particularly in high-dimensional response settings.

There has been a lot of recent interest in variable selection problems for high-dimensional linear regression models because of the increasingly frequent and important in diverse fields of economics,  finance and machine learning.
For univariate (or single) response case (i.e., $p=1$), a variety of methods have been developed. This includes the penalty-based methods such as the least angle and shrinkage selection operator (LASSO, \cite{Tibshirani96R}), the
adaptive LASSO \citep{Zou06A}, the smoothly clipped absolute deviation (SCAD \cite{FanL01V}), the
 minimax convex penalty (MCP, \cite{Zhang10N}); the screening-based methods such as the sure independence screening (SIS, \cite{FanL08S}),  the covariate assisted screening estimates (CASE, \cite{KeJ14C}); the testing based methods such as the multiple testing approach by the false discovery rate (FDR) \citep{LiuL14H,XiaC18T} and many other related methods.
We refer to some recent review papers \citep{Shao97A,FanL10S,HuangB12S,AnzanelloF14R,HeinzeW18V,Desboulets18R,LeeC19S,CaiG23S} for more details. However,  there is comparatively less literature available for multiple responses (i.e. $p>1$). \cite{Xia17T} proposed a row-wise multiple testing procedure when $p$ is fixed; \cite{KongL17I} suggested a screening method via the distance correlations of the responses and each covariate for high-dimensional multi-response interaction models. For $p\to\infty$, following \cite{BaiF14S}, \cite{BaiC22A} investigated the asymptotic properties of the classical AIC, BIC and $C_p$ criteria; and \cite{SakuraiF20E,OdaY20F} established the consistencies of the KOO methods with AIC, BIC and $C_p$ thresholds under normality errors.

Main contributions of this paper are: (1) We obtain the asymptotic distributions of the KOO statistics $\cK_j$ for any $j=1,\dots,k$ under some mild moment conditions and 3L asymptotic framework:  large-response ($p\to\infty$), large-model ($k\to\infty$) and large-sample ($n\to\infty$). These theoretical results are applicable to many other model selection rules, such as growth curve model, multiple discriminant analysis,
principal component analysis,  canonical correlation analysis, and graphical model (e.g., \cite{FujikoshiS19C,OdaS20C,FujikoshiS23H}). (2) A scalable model selection method based on the KOO statistics is proposed. In practice, we use a multiplier bootstrap procedure to estimate the asymptotic thresholds. Simulation studies and real data analyses suggest the proposed model selection method performs favorably against the existing KOO methods with AIC, BIC and $C_p$ thresholds.

The remainder of this paper is organized as follows. In Section \ref{mainresults}, we state the main results of this paper, which include the almost sure limit and central limit theorem (CLT) of the KOO statistics. In Section \ref{selectionrule}, we propose a model selection method for the high-dimensional linear regression model based on the KOO statistics and information criteria.  In Sections \ref{simulation} and \ref{realdata},  we conduct some simulation studies and real data analysis, respectively.  Proofs of the main theorems under normality are given in Section  \ref{proofnorm}  since they are less technical and of independent interests. Proofs for general error distributions using random matrix theory are provided in the Appendix for interested readers.

%

\section{KOO statistics}\label{mainresults}
\subsection{Notation and preliminary}
We begin this section with some basic notation and definitions. In this paper, matrices and vectors are denoted by boldface uppercase and lowercase letters, respectively.
Let   $\bbI_n$  denote the identity matrix of order $n$,
\begin{equation}
\label{nhatS}
 \widehat\bgS_{\bbj}=\frac1n\bbY'\bbQ_{\bbj}\bbY,~~\bbQ_{\bbj}=\bbI_n-\bbP_{\bbj},~~
\bbP_{\bbj}=\bbX_{\bbj}(\bbX_{\bbj}'\bbX_{\bbj})^{-1}\bbX_{\bbj}',
\end{equation}
 $|{\bbj}|$  the cardinality of subset $\bbj$, and $|\widehat \bgS_{\bbj}|$  the determinant $\widehat \bgS_{\bbj}$.
Note that
 $\bbP_{\bbj}$ is an orthogonal projection of rank $|{\bbj}|$ onto the subspace spanned by $\bbX_{\bbj}$, and $\bbQ_{\bbj}$ is the orthogonal projection of rank $n-|{\bbj}|$ onto the
orthogonal complement subspace spanned by $\bbX_{\bbj}$. For brevity, we suppress the subscript $[k]$ for full model, and denote the true model subscript by $*$ and the subscript of model ${[k]} \backslash j$ by $j$ (e.g., $\bbQ:=\bbQ_{[k]}$, $\bbQ_{\bbj_*}:=\bbQ_{*}$ and $\bbQ_j:=\bbQ_{[k]\backslash j}$).
The identity matrix, all-zero matrix, all-one vector and all-zero vector, whose orders are often clear from the context and thus will not be indicated, are denoted by $\bbI$, $\bbO$, $\bf1$, and $\bf0$, respectively.
We call $j$ (or variable $\bbx_j$)  true if $j \in \bbj_*$, and $j$ (or variable $\bbx_j$) is spurious if $j \notin \bbj_*$.
For a matrix $\bbA$, its spectral norm and maximum norm are denoted by $\|\bbA\| $ and $\|\bbA\|_\infty$, respectively. The largest and smallest eigenvalues of $\bbA$ are denoted by $\lambda_{\max}^\bbA$ and $\lambda_{\min}^{\bbA}$, respectively. For two matrices $\bbA$ and $\bbB$ of the same dimension, $\bbA\circ \bbB$ stands for the Hadamard product of  $\bbA$ and $\bbB$. We denote the probability by $\P$, the   expectation by $\E$, and the trace by $\tr$. Define $ c_n:=p/n$ and $\alpha_n:= k/n$. Throughout this paper, we use $o(1)$ (respectively, $o_p(1)$, $o_{a.s.}(1)$) to denote (respectively, in probability, almost surely) scalar negligible entries. And the notations $O(1)$, $O_p (1)$ and $O_{a.s.} (1)$ are used in a similar way.

We now introduce  the KOO statistics
\begin{align*}
 \cK_j=\tr(\widehat \bgS^{-1}\widehat \bgS_j)-p.
\end{align*}
%
It is known that for testing
$\bm\theta_j=\bf 0$ under normality, 
the Lawley-Hotelling trace statistic can be expressed as
$(n-k)(\cK_j+p)$. Next  we will investigate the statistical properties of  $\cK_j$ under  the 3L asymptotic framework:  large-model ($k$), large-sample ($n$) and large-dimensional response ($p$). Before presenting our main theoretical results, we briefly analyze the statistic $\cK_j$. 
Let
\begin{gather*}
 \mathbf{a}_{j}=\mathbf{Q}_j\mathbf{x}_{j} /\left\|\mathbf{Q}_j \mathbf{x}_{j}\right\|.
\end{gather*}
By Sylvester's determinant theorem, we have that
\begin{align}\label{det_eq}
	n\widehat\bgS_j
=n\widehat\bgS+\bbY'\bba_{j}\bba_{j}'\bbY
\end{align}
which implies
\begin{align}\label{Cpj}
\cK_j=n^{-1}\mathbf{a}_{j}' \mathbf{Y}  \widehat\bgS^{-1}\mathbf{Y}' \mathbf{a}_{j}.
\end{align}
If we plug the model \eqref{fullmodel} into the $j$th KOO statistic, we have
\begin{align*}
 \cK_j=(\bba_{j}'\bbX_*\bm\Theta_*\bgS^{-1/2}+\bba_{j}'\bbE)(\bbE'\bbQ\bbE)^{-1}(\bbE'\bba_{j}+\bgS^{-1/2}\bm\Theta_*'\bbX_*'\bba_{j}).
\end{align*}
When $j$ is spurious (i.e., $j\notin\bbj_*$), $\bba_{j}$ and $\bbX_*$ are orthogonal. Thus, in this case,
\begin{align*}
 \cK_j= \bba_{j}'\bbE(\bbE'\bbQ\bbE)^{-1}\bbE'\bba_{j}.
\end{align*}
On the other hand, when $j$ is true (i.e., $j\in\bbj_*$), then 
\begin{align*}
 \cK_j\asymp \bba_{j}'\bbE(\bbE'\bbQ\bbE)^{-1}\bbE'\bba_{j}+\bbx_{j}'\bbQ_j\bbx_{j}\bmth_j'\bgS^{-1/2}(\bbE'\bbQ\bbE)^{-1}\bgS^{-1/2}\bmth_j.
\end{align*}
We emphasize that, for spurious $j$, the KOO statistics $ \cK_j$ are independent of the population covariance matrix $\bgS$. This property is highly desirable as it eliminates the involvement of unknown parameters. Furthermore, the term $\bbx_{j}'\bbQ_j\bbx_{j}\bmth_j'\bgS^{-1/2}(\bbE'\bbQ\bbE)^{-1}\bgS^{-1/2}\bmth_j>0$ becomes a key indicator to distinguish between spurious and true variables, with its value serving as a crucial factor in the determination process. The detailed discussion is stated in the next subsection.

\subsection{Asymptotical properties of the KOO statistics}
In this subsection, we state the asymptotics of the KOO statistics and illustrate how the KOO statistics of true variables behave differently from that of the KOO statistics of spurious variables under some mild conditions.
Before stating these results, we collect the needed conditions below. 
\begin{itemize}
\item[(C1)] As $\min\{k,p,n\}\to\infty$, $ c_n \to c \in (0,1)$ and $\alpha_n \to \alpha \in [0,1) $ satisfying $\alpha +c <1$.
\item[(C2)]  The true model $\bbj_* \subset [k]$, and  $|\bbj_*|$ is allowed to diverge as $k\to\infty$.
\item[(C3)]  The entries $e_{ij}$ of $\bbE$ are independent and identically distributed (i.i.d.) with zero means, unit
 variances, and finite fourth moments, i.e., $\tau=\E e_{ij}^4-3\in ( -\infty,\infty)$.
 \item[(C4)]  Matrix $\bbX'\bbX$ is positive definite for all $n >k+p$.
\end{itemize}
Our main results of this paper are stated below. The proofs, under normality of errors,  will be given in Section \ref{proofnorm}; and the general proofs without assuming normality of errors will be given in the Appendix.

Let
\begin{equation}
\label{delta-j}
\delta_j:=\delta_{nj}=p^{-1}\bbx_{j}'\bbQ_j\bbx_{j}\bmth_j'\bgS^{-1}\bmth_j .
\end{equation}
The following theorem identifies the strong limits of the KOO statistics $\cK_j$ for all $j\in[k]$.
\begin{thm}\label{limit1}
	Under conditions {\rm (C1) -- (C4)},  we have uniformly in $j\in[k]$,
$$
\cK_j= \begin{cases}
\frac{c_n}{1-c_n-\alpha_n} + o_{a.s.}(1), & \mbox{~if~}j\notin\bbj_*, \\
(1+\delta_j) \left[ \frac{c_n}{1-c_n-\alpha_n} + o_{a.s.}(1)\right], & \mbox{~if~}j\in\bbj_*.
\end{cases}
$$

\end{thm}


As  $\bm\theta_j$  and $\bgS$ are typically unknown in practice,  the limits of   $\cK_j$'s  for   $j\in\bbj_*$ are unknown.  However, the fluctuations of the $\cK_j$'s for spurious variables are pretty simple, which is described in the following theorem.
\begin{thm}\label{clt1}
	Under  conditions {\rm (C1) -- (C4)}, for any fixed integer $q>0$ and $\{j_1,\dots,j_q
	\}\subset\bbj\backslash\bbj_*$, the random vector
	\begin{align*}
		\sqrt{p}\bbG_q^{-1/2}\left[(\cK_{j_1},\dots,\cK_{j_q})'-\frac{c_n}{1-c_n-\alpha_n}{\bf1}_q\right]
	\end{align*}
converges weakly to the standard $q$-dimensional Gaussian random vector, where
$$\bbG_q=\frac{c_n^2}{(1-\alpha_n-c_n)^2} \left[\frac{2(1-\alpha_n)}{(1-\alpha_n-c_n)}(\cA_q'\cA_q)^2+\tau(\cA_q\circ\cA_q)'(\cA_q\circ\cA_q)\right], $$ and
 $\cA_q=(\bba_{j_1},\dots,\bba_{j_q})$ is an $n\times q$ non-random matrix.
\end{thm}

Theorem \ref{clt1} is of independent interest:
As $\cK_j$'s are the {basic  statistics} for
 testing the hypothesis that $\bm\theta_j=\bf 0$,  this theorem can be used to obtain the CLTs of these statistics under the null hypothesis.
  Moreover, 	if $\tau=0$ (e.g.,  $\{e_{ij}\}$ come from a standard normal distribution), then the second term in   $\bbG_q$  vanishes; or if $\max_{j\in\bbj\backslash\bbj_*}\|\bba_j\|_\infty=o(1)$,  then the second term in the covariance matrix $\bbG_q$  tends to 0 as $n\to\infty$.

When $\tau\neq 0$, we propose an estimator of $\tau$,
$$\hat \tau=\left\{p^{-1}\tr[(\bbY'\bbQ\bbY-(n-k)\bbI)\circ(\bbY'\bbQ\bbY-(n-k)\bbI)]-2(n-k)\right\}/\tr(\bbQ\circ\bbQ),$$  which is shown to be unbiased and weakly consistent in Theorem \ref{hattau} below.

\begin{thm}\label{hattau}
	Under the conditions {\rm (C1) -- (C4)},
 $ \hat\tau$ is an unbiased and weakly consistent estimator of $\tau$.
\end{thm}

Combining Theorems \ref{clt1} and \ref{hattau},  the rejection region of the KOO statistics for testing whether some variables are spurious can be constructed.  However, in order to know the power, we also need to know the fluctuations for the statistics of the true variables. The following theorem states that under some additional assumptions, the KOO statistic of the true variable is comparable to that of the spurious variables.


\begin{thm}\label{clt2}
In addition to the conditions {\rm (C1) -- (C4)}, for $j\in\bbj_*$, we assume that
\begin{enumerate}
\item[(C5):] $\E e_{11}^3=0$.
\item[(C6):]As $\min\{p,n,k\}\to\infty$, $\|\bba_j\|_\infty=o(1)$, $\bbx_{j}'\bbQ_j\bbx_{j}\|{\bm\theta_j}'\bgS^{-1/2}\|^2_\infty=o(p)$.
\item[(C7):] As $\min\{p,n,k\}\to\infty$, $\delta_j$ tends to a constant.
\end{enumerate}
 Then,
\begin{align*}
 \sqrt{p}\left(\cK_j-\frac{c_n(1+\delta_j)}{1-c_n-\alpha_n}\right)/\sigma_{nj}\dto N(0,1),
\end{align*}
where $\sigma_{nj}^2={2c_n^2[(1-\alpha_n)(1+2\delta_j)+c_n\delta_j^2]}/{(1-\alpha_n-c_n)^3}$.
%
%
\end{thm}

\subsection{Some remarks on  the theorems}
\begin{remark}
The condition,  $c>0$, in \rm (C1) is due to technical reasons: our main tools are from random matrix theory (RMT) and RMT generally assumes the limit $p/n$ exists and is positive.
Note further that we make no explicit use of the unknown limits $\alpha$ and $c$ in all the theorems below. Rather,  we used $\alpha_n$ and $c_n$,  which are always positive, in our results.
	
\end{remark}
\begin{remark}
If the model size $k$ is greater than the sample size $n$ but the true model size $k_*$ is fixed, one can first apply screening methods (such as the sure independence screening method based on the distance correlation \citep{LiZ12F}, and interaction pursuit via distance correlation \citep{KongL17I}) to ensure condition (C1) holds. For further details on the screening methods, see  \citep{FanL08S,FanL10S}.
\end{remark}
\begin{remark}
	If the entries $ e_{ ij}$ of $\bbE$ are independent but not necessarily
	 identically distributed,  our results in this paper continue to hold provided an additional Lindeberg-type condition:
 \begin{align*}
  \frac{1}{\eta^{4} n^{2}} \sum_{i, j} \mathbb{E}\left[\left|e_{i j}\right|^{4}\mathbbm{1}\left\{\left|e_{i j}\right| \geq \eta \sqrt{n}\right\}\right]=o(1),
\end{align*}
for any $\eta>0$.  Here, $\mathbbm{1}\{\cdot\}$ stands for the indicator function. The proofs are analogous but slightly more tedious, and we do not pursue this extension in this paper.
\end{remark}

\begin{remark}
From Theorems \ref{clt1} and \ref{clt2}, we can theoretically investigate the asymptotic power of whether a variable is spurious.   However, for testing whether a variable is true, the asymptotic distribution of the true KOO statistic  (i.e., Theorem \ref{clt2}) cannot be applied directly since $\delta_j$ is unknown when $j$ is a true variable. 
Variable selection problem will be discussed in the next section in detail.
%
\end{remark}
\section{Selection criteria based on the KOO statistics}\label{selectionrule}

Theorem \ref{limit1} highlights the crucial role of  $\delta_j$  in differentiating the true variables from the spurious ones. For spurious variables,  $\cK_j$'s should be close to the point ${c_n}/{(1-c_n-\alpha_n)}$  when $n,p,k$ are large. Since $\delta_j$ is always positive for $j\in\bbj_*$, the true variables would be separated from ${c_n}/{(1-c_n-\alpha_n)}$ and thus can be identified by the largest  $\cK_j$'s. Moreover, we can deduce a strongly consistent estimator for the true variables from this theorem.
Let
\begin{gather*}
 \hat \bbj_\vartheta=\left\{j\in{[k]}|\cK_{j}>\frac{c_n(1+\vartheta)}{1-\alpha_n-c_n}\right\},~~\vartheta>0.
\end{gather*}
Then, we have the following corollary of Theorem \ref{limit1}.
\begin{corollary}\label{gen_kooth}
Assume that  conditions {\rm (C1) -- (C4)} hold and   $\lim\delta_j>0$  for all $j\in \bbj_{*}$.
	Then, for any fixed value $\vartheta\in(0,\min_{j\in\bbj_*}\{\lim\delta_j\})$,  
	\begin{align*}
 \lim_{n,p\to\infty}\hat \bbj_\vartheta\asto\bbj_*.
\end{align*}
	\end{corollary}
\begin{remark}
This corollary implies the strong consistency for the KOO methods with AIC, BIC and $C_p$ thresholds if $\delta_j$ satisfies the conditions.
\end{remark}

In practice, however, choosing a suitable    $\vartheta$ is important but very challenging because (1) the largest spurious KOO statistic may converge to its limit slowly; (2) the spurious KOO statistics are correlated; and (3) the limits of the true KOO statistics are unknown.  Hence, we propose a high-dimensional multiplier bootstrap procedure to approximate the distribution of the largest spurious KOO statistic $\cK_j$, from which a selection criterion for the linear regression model \eqref{fullmodel} under the 3L framework is formulated.


 Denote the estimator of the true model be
 \begin{align*}
 \hat\bbj_*=\{j\in[k]:\cK_j>K_\nu\},
\end{align*}
where $K_\nu$ is the critical value with at significance level $\nu$, which is estimated by  Algorithm 1.

\begin{algorithm}[!ht]\label{alg1}
\DontPrintSemicolon

\KwInput{$\nu$,  $\bbY$, $\bbX$ and estimator $\hat\tau$ based on $\{\bbY,\bbX\}$}
    \KwOutput{Estimator      $\hat{K}_\nu$}

Compute $\cA_k=(\bba_{1},\dots,\bba_{k})$. 

Generate a random matrix $\tilde\bbE$ with $n\times p$ i.i.d. zero mean, unit variance and $\hat\tau$ excess kurtosis elements.


Compute $\bbK=\cA_k'\tilde\bbE(\tilde\bbE'\bbQ\tilde\bbE)^{-1}\tilde\bbE'\cA_k$.

Compute the largest value of the diagonal elements of $\bbK$ and denote it by $\tilde\cK^{(1)}$.

Repeat $N$ times of the above procedures 2--4, and obtain $\{\tilde\cK^{(1)},\dots,\tilde\cK^{(N)}\}$.

Compute the $100(1 -\nu)$th quantile of $\{\tilde\cK^{(1)},\dots,\tilde\cK^{(N)}\}$  and  denote it by $\hat{K}_\nu$.

\caption{Estimation of $K_\nu$ }
\end{algorithm}


From Theorem \ref{clt1},  the critical value $K_\nu$ may depend on $\|\bba_i\|_\infty$ or the excess kurtosis but not on the exact distribution of the errors.   The boxplots of the spurious KOO statistics $\cK_j$'s for different distributions presented in Fig. \ref{difdis} support this claim.
 In this simulation, we set  $\bm\Theta=\bbO$, $\boldsymbol{\Sigma}=\bbI$ and generate two predictor matrices: the first one is a $2000\times 600$ matrix with i.i.d. entries from $
 U(1,5)
 $; and the second one is a $2000\times 600$ diagonal matrix.  As the values of the diagonal elements do not affect the result, the diagonal entries were chosen to be 1 in our simulation. We examine six different distributions of the errors: standard normal distribution $N(0,1)$,   standardized uniform distribution $U(0,1)$, standardized Bernoulli distribution $B(1,\rho)$ with parameter $\rho=(6-\sqrt{6})/12$, standardized chi-square distribution with 12 degrees of freedom $\chi^2(12)$,  standardized $t$-distribution with 10 degrees of freedom $t_{10}$, standardized Poisson distribution with parameter 1 $Pois(1)$, standardized exponential distribution with rate parameter 1 $Exp(1)$ and standardized chi-square distribution with 2 degrees of freedom $\chi^2(2)$. Note that $\|\bba_i\|_\infty \to0$ for the random predictor matrix, $\|\bba_i\|_\infty =1$ for the rectangular diagonal predictor matrix, the excess kurtosis of $N(0,1)$ is 0, the excess kurtoses of $Exp(1)$ and $\chi^2(2)$ are 2, the excess kurtoses of $\chi^2(12)$, $t_{10}$ and $Pois(1)$ are 1, and the excess kurtoses of $U(0,1)$ and $B(1,(6-\sqrt{6})/12)$ are $-6/5$. Hence, in practice for convenience, we can use standardized $\chi^2$ distribution with $12/\hat\tau$ degrees of freedom if $\hat\tau>0$ and standardized Bernoulli distribution $B(1,\rho)$ with parameter $\rho$ satisfying $\rho(1-\rho)=1/(6-\hat\tau)$ if $\hat\tau<0$. Of course, if $\max_i\|\bba_i\|_\infty\to0$, we can use the standard normal distribution directly.
\begin{figure}[htbp!]
\center
\subfigure[Random predictor matrix]{
		\includegraphics[width=7.2cm,height=5cm]{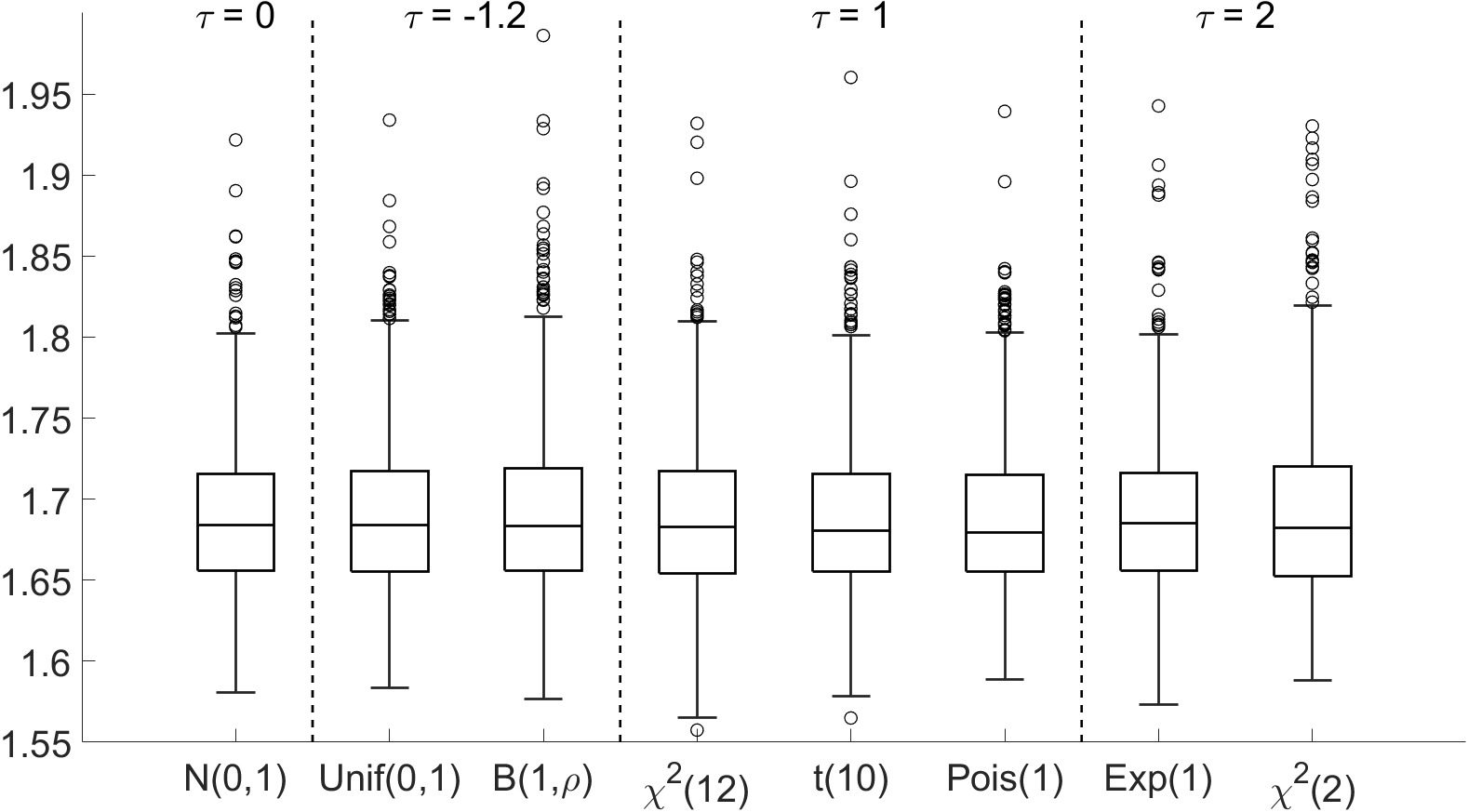}}
		\subfigure[Rectangular diagonal predictor matrix]{
		\includegraphics[width=7.2cm,height=5cm]{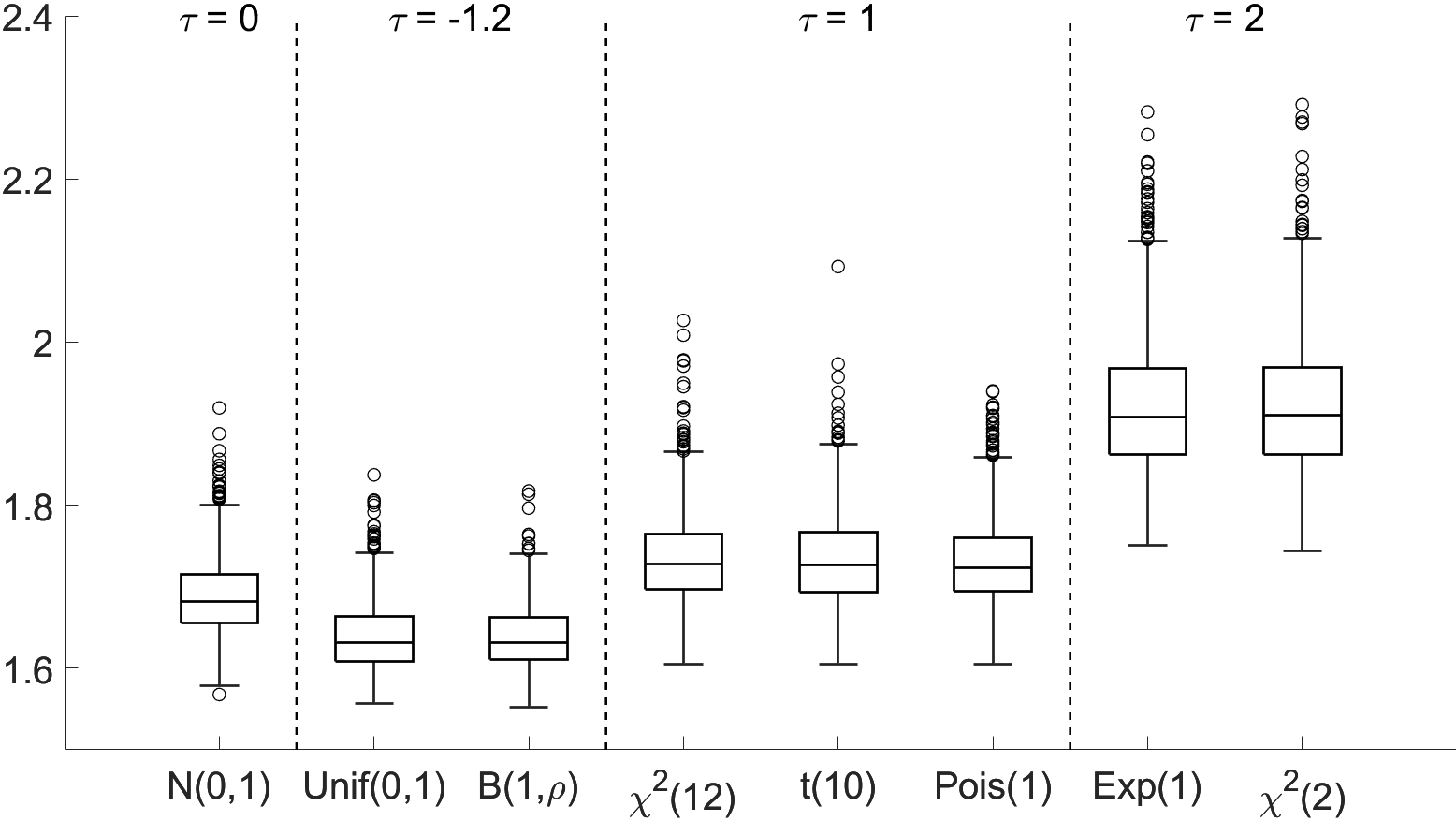}}
		\caption{Boxplots of the spurious KOO statistics
		 $\{\cK^{(j)}$, ~$j=1,\dots,1000\}$ with six different normalized distributions and  two  predictor matrices. The y-axis represents the values of $\cK^{(j)}$'s.}
		\label{difdis}
\end{figure}


\section{Simulation studies}\label{simulation}
In this section, we numerically examine the properties of the proposed KOO method in a 3L framework with different settings. For comparison, we also report the results of KOO methods with AIC, BIC and $C_p$ thresholds as proposed by \cite{NishiiB88S} and implemented by \cite{FujikoshiS19C,OdaS20C,NakagawaW21K,Fujikoshi22H}. Specifically, the KOO methods with AIC, BIC and $C_p$ thresholds, respectively, choose the model
\begin{eqnarray*}
 \hat \bbj_*^{A}&=&\{j\in{[k]}:\log(1+\cK_j)>2c_n\}, \\
 \hat \bbj_*^{B}&=&\{j\in{[k]}:\log(1+\cK_j)>\log(n)c_n\},\\
 \hat \bbj_*^{C}&=&\{j\in{[k]}:(1-\alpha_n)\cK_j>2c_n\}.
\end{eqnarray*}
For simplicity, we abbreviate the KOO method with our bootstrapping threshold to KBT, the KOO method with AIC threshold as KAIC. Similar abbreviations KBIC and KCp are used.
We consider the following two settings.
\begin{itemize}
	\item[]Setting I: Fix $k_*=5$, $p/n=\{0.2, 0.4\}$ and $k/n=\{0.2,0.4\}$ with  $n=100, 500, 1000, 2000$.  The results for $n=2000$ are given in the Appendix.
	Set $\boldsymbol{\Sigma}=\bbI$, $\bbX=(x_{ij})_{n\times k}$, $\Theta_{\bbj_*}={\bf 1}_{5}{\bm\theta}_*$ and $\Theta=(\Theta_{\bbj_*},\bf{0})$, where $\{x_{ij}\}$ are i.i.d. generated from the continuous uniform distributions $U(1,5)$, ${\bf 1}_{5}$ is a five-dimensional vector of ones and ${\bm\theta}_*=((-0.5)^0,\dots,(-0.5)^{p-1})$.

		\item[]Setting II: Same as Setting I, except $\bbX=(\bbI_k,\bbO_{k\times (n-k)})'$ and $\Theta_{\bbj_*}=\sqrt{n}{\bf 1}_{5}{\bm\theta}_*$.
\end{itemize}

For Setting I, we consider three cases for the distribution of $\bbE$:
\begin{itemize}
\item[(i)] Standard normal distribution, $e_{ij} \sim N(0,1)$;
\item[(ii)] Standardized $t$ distribution with three degrees of freedom, i.e., $e_{ij} \sim t_3/\sqrt{5/3}$;
\item[(iii)] Standardized chi-square distribution with two degrees of freedom, i.e., $e_{ij} \sim (\chi^2(3)-3)/\sqrt{6}$.
\end{itemize}
 Since  $\|\bba_j\|_\infty\to0$ in  Setting I, we use $\tilde\bbE$ with the standard normal distribution to estimate $\hat{K}_\nu$. We emphasize that the excess kurtosis of distribution $t_3$ is infinite.

For Setting II, we consider three cases for the distribution of $\bbE$:
\begin{enumerate}
	\item[(iv)] Standardized exponential distribution with rate parameter 1, i.e.,  $e_{ij} \sim Exp(1)-1$;
	\item[(v)] Standardized Poisson distribution with parameter 1, i.e., $e_{ij} \sim Pois(1)-1$;
	\item[(vi)] Standardized uniformly distribution, i.e., $e_{ij} \sim U(-\sqrt{3},\sqrt{3}) $.
\end{enumerate}
Since   $\|\bba_j\|_\infty=1$ in Setting II, we use $\tilde\bbE$ with standardized $\chi^2$ distribution and standardized Bernoulli distribution, respectively, to estimate $\hat{K}_\nu$ with some suitably chosen parameter values. 

In all the simulation studies, we choose two critical points in the KOO methods:.
$$
 \hat\bbj_*^{(0)}=\{j\in[k]:\cK_j>\hat{K}_0\} \quad \mbox{ and } \quad
 \hat\bbj_*^{(5)}=\{j\in[k]:\cK_j>\hat{K}_{0.05}\},
$$
where $\hat{K}_0$ and $\hat{K}_{0.05}$ are the largest and the $95$th percentile of 1,000 bootstrap values, respectively.

We first explain our choices of the settings and the distributions. Since  the KOO criteria depend on the values $\delta_j=p^{-1}\bbx_{j}'\bbQ_j\bbx_{j}\theta_j'\bgS^{-1}\theta_j $, it suffices to set $\bgS=\bbI$ and vary $\bm\Theta_*$ and $\bbX$  in conducting our simulation studies. Settings I and II both ensure $\delta_j$ are bounded above. For the case   $\delta_j\to\infty$, the KOO statistics for the true variables and spurious variables are well separated, and all the compared selection methods will not show significant differences.
The selection of distributions comprises five continuous distributions and one discrete distribution. The distribution described in (ii) only has finite second moment. This selection was made to investigate the implications of not satisfying the condition of finite fourth moment.
To measure in greater detail the performance of these selection rules, the numbers of times, in 1000 repetitions, a selection rule under-specifies the true model, exactly identifies it and over-specifies it were tabulated. 
When the selection rule over-specifies the true model, we also report the average number of spurious variables selected in the last row of each sub-table. Due to space consideration,   we  present  selected results, but typical,  of Setting I (i) and Setting II (iv) in Tables \ref{KOO1} and \ref{KOO4}, respectively. Full set of results, including those for $n=2000$, can be found in the Appendix.

 \begin{table}[htbp]
 \setlength{\tabcolsep}{2.8pt}
 \center
	\begin{tabular}{|c|ccccc|ccccc|ccccc|}
		\hline
\multicolumn{16}{|c|}{$\alpha=0.2$, $c=0.4$} \\
\cline{1-16}
		 &\multicolumn{5}{c|}{$n=100$} &\multicolumn{5}{c|}{$n=500$} &\multicolumn{5}{c|}{$n=1000$} \\
		 \cline{2-16}
		& ${\hat\bbj_*^{A}}$& $\hat\bbj_*^{B}$&$\hat\bbj_*^{C}$&$\hat\bbj_*^{(0)}$&$\hat\bbj_*^{(5)}$&${\hat\bbj_*^{A}}$& $\hat\bbj_*^{B}$&$\hat\bbj_*^{C}$&$\hat\bbj_*^{(0)}$&$\hat\bbj_*^{(5)}$&${\hat\bbj_*^{A}}$& $\hat\bbj_*^{B}$&$\hat\bbj_*^{C}$&$\hat\bbj_*^{(0)}$&$\hat\bbj_*^{(5)}$\\

\hline
U-S & 0 & 938 & 0 & 640 & 19 & 0 & 1000 & 0 & 0 & 0 & 0 & 1000 & 0 & 0 & 0\\
\hline
T-S & 35 & 62 & 0 & 360 & 940 & 2 & 0 & 0 & 1000 & 953& 23 & 0 & 0 & 998 & 957 \\
\hline
O-S& 965 & 0 & 1000 & 0 & 41 & 998 & 0 & 1000 & 0 & 47 & 977 & 0 & 1000 & 2 & 43\\

A-S& 3.69 & -- & 7.06 & -- & 1.05 & 6.86 & -- & 46.30 & -- & 1.04 & 3.92 & -- & 95.28 & 1 & 1.05\\

		\hline
				\hline
\multicolumn{16}{|c|}{$\alpha=0.4$, $c=0.2$} \\
\cline{1-16}
		 &\multicolumn{5}{c|}{$n=100$} &\multicolumn{5}{c|}{$n=500$} &\multicolumn{5}{c|}{$n=1000$} \\
		 \cline{2-16}
		& ${\hat\bbj_*^{A}}$& $\hat\bbj_*^{B}$&$\hat\bbj_*^{C}$&$\hat\bbj_*^{(0)}$&$\hat\bbj_*^{(5)}$&${\hat\bbj_*^{A}}$& $\hat\bbj_*^{B}$&$\hat\bbj_*^{C}$&$\hat\bbj_*^{(0)}$&$\hat\bbj_*^{(5)}$&${\hat\bbj_*^{A}}$& $\hat\bbj_*^{B}$&$\hat\bbj_*^{C}$&$\hat\bbj_*^{(0)}$&$\hat\bbj_*^{(5)}$\\
\hline
U-S & 0 & 42 & 0 & 828 & 41 & 0 & 129 & 0 & 0 & 0 & 0 & 729 & 0 & 0 & 0\\
\hline
T-S & 0 & 923 & 3 & 172 & 919 & 0 & 871 & 0 & 998 & 965 & 0 & 271 & 41 & 1000 & 954\\
\hline
O-S & 1000 & 35 & 997 & 0 & 40 & 1000 & 0 & 1000 & 2 & 35 & 1000 & 0 & 959 & 0 & 46\\
A-S & 16.50 & 1.09 & 6.67 & -- & 1.12 & 100.87 & -- & 8 & 1 & 1& 213.52 & -- & 3.28 & -- & 1.02 \\
\hline
	\end{tabular}
	\caption{ Selection times of the KOO methods with AIC, BIC, $C_p$ thresholds and bootstrap methods under Settings  (I)  and (i) based on 1,000 replications. Here U-S, T-S, O-S and A-S stand for number of times a selection method under-specified the true model,  number of times a selection method identified the true model exactly, number of times a selection method over-specified the true model,  and  the average number of spurious variables a selection method identified when it over-specified the model, respectively.
	}\label{KOO1}
\end{table}

 \begin{table}[htbp]
 \setlength{\tabcolsep}{2.8pt}
 \center
	\begin{tabular}{|c|ccccc|ccccc|ccccc|}
		\hline
\multicolumn{16}{|c|}{$\alpha=0.2$, $c=0.4$} \\
\cline{1-16}
		 &\multicolumn{5}{c|}{$n=100$} &\multicolumn{5}{c|}{$n=500$} &\multicolumn{5}{c|}{$n=1000$} \\
		 \cline{2-16}
		& ${\hat\bbj_*^{A}}$& $\hat\bbj_*^{B}$&$\hat\bbj_*^{C}$&$\hat\bbj_*^{(0)}$&$\hat\bbj_*^{(5)}$&${\hat\bbj_*^{A}}$& $\hat\bbj_*^{B}$&$\hat\bbj_*^{C}$&$\hat\bbj_*^{(0)}$&$\hat\bbj_*^{(5)}$&${\hat\bbj_*^{A}}$& $\hat\bbj_*^{B}$&$\hat\bbj_*^{C}$&$\hat\bbj_*^{(0)}$&$\hat\bbj_*^{(5)}$\\

\hline
U-S & 0 & 925 & 0 & 991 & 622 & 0 & 1000 & 0 & 2 & 0 & 0 & 1000 & 0 & 0 & 0 \\
\hline
T-S & 2 & 74 & 0 & 9 & 361 & 0 & 0 & 0 & 997 & 934 & 0 & 0 & 0 & 1000 &938\\
\hline
O-S & 998 & 1 & 1000 & 0 & 17 & 1000 & 0 & 1000 & 1 & 66 & 1000 & 0 & 1000 & 0 & 62\\

A-S & 4.74 & 1 & 7 & -- & 1.06 & 16.40 & -- & 45.88 & 1 & 1 & 18.74 & -- & 95.36 & -- & 1.03\\
\hline
\hline

\multicolumn{16}{|c|}{$\alpha=0.4$, $c=0.2$} \\
\cline{1-16}
		 &\multicolumn{5}{c|}{$n=100$} &\multicolumn{5}{c|}{$n=500$} &\multicolumn{5}{c|}{$n=1000$} \\
		 \cline{2-16}
		& ${\hat\bbj_*^{A}}$& $\hat\bbj_*^{B}$&$\hat\bbj_*^{C}$&$\hat\bbj_*^{(0)}$&$\hat\bbj_*^{(5)}$&${\hat\bbj_*^{A}}$& $\hat\bbj_*^{B}$&$\hat\bbj_*^{C}$&$\hat\bbj_*^{(0)}$&$\hat\bbj_*^{(5)}$&${\hat\bbj_*^{A}}$& $\hat\bbj_*^{B}$&$\hat\bbj_*^{C}$&$\hat\bbj_*^{(0)}$&$\hat\bbj_*^{(5)}$\\
\hline
U-S & 0 & 4 & 0 & 999 & 597 & 0 & 7 & 0 & 7 & 0 & 0 & 27 & 0 & 0 & 0\\
\hline
T-S & 0 & 348 & 0 & 1 & 386 & 0 & 993 & 0 & 993 & 961 & 0 & 973 & 0 & 1000 & 939\\
\hline
O-S & 1000 & 648 & 1000 & 0 & 17 & 1000 & 0 & 1000 & 0 & 39 & 1000 & 0 & 1000 & 0 & 61\\

A-S & 15.31 & 1.67 & 9.23 & -- & 1 & 94.64 & -- & 28.61 & -- & 1 & 198.92 & -- & 31.12 & -- & 1.03\\
  \hline

	\end{tabular}
	\caption{ Selection times of the KOO methods with AIC, BIC, $C_p$ thresholds and bootstrap methods under Settings (II) and (iv) based on 1,000 replications. Here U-S, T-S, O-S and A-S stand for number of times a selection method under-specified the true model,  number of times a selection method identified the true model exactly, number of times a selection method over-specified the true model,  and  the average number of spurious variables a selection method identified when it over-specified the model, respectively.
	}\label{KOO4}
\end{table}

Based on our simulation results, the following observations are made: (1) The proposed KBT  are the most robust among the compared methods, especially when the sample size $n$ is large. (2) If the sample size $n$ is small, we recommend choosing a bigger $\nu$ in order to avoid missing the true variables. After all, admitting a small number of spurious variables is a better tradeoff  than missing some true variables.  (3) Choosing a bigger $\nu$ may select more spurious variables, but unlike the KAIC and KCp , the number of spurious variables selected is still under control.  (4) The simulation results are very similar across different distributions of errors, which suggests these selection rules are rather robust against the distributions of errors.  (5) When $\max_i\|\bba_i\|_\infty\to0$, our proposed methods also work well even the finite fourth moment condition does not hold, suggesting  that our theorems continue to hold even under weaker conditions. Our guess is that finite second moment of the underlying error distributions is enough.  (6) The performances of  KAIC, KBIC and KCp are not acceptable under our settings: KAIC and KCp frequently over-specify the true models quite substantially,  and KBIC frequently under-specifies the true models. Under some special cases,  KBIC has good selection times,  however, KBT in general outperforms KBIC.

\section{Real data analysis}\label{realdata}
We apply the proposed methods to two real examples.
The first example is a chemometrics data taken from \cite{SkagerbergM92M} (we replaced the value 19203 with 1.9203 in the 37th observation).
The data are taken from a simulation of a low-density tubular polyethylene reactor studying the relationship between polymer properties and the process. The predictor variables consist of 20 temperatures measured at equal distances along the low-density polyethylene reactor section, together with the wall temperature of the reactor and the solvent feed rate. The responses are the output characteristics of the polymers, including two molecular weights, two branching frequencies and the contents of two groups. This data set has been studied by \cite{BreimanF97P} and \cite{SimilaT07I}.  Similar to \cite{BreimanF97P}, we log-transformed the response values because they are highly skewed to the right. In total,  there are $n=56$ observations with $k=22$ predictor variables and $p=6$ responses.

 We present the scatterplot of  $\{\cK_{j}\}$ in  descending order in Figure \ref{real_fig1}.  We also indicate the critical values of KAIC, KBIC and KCp, and $\hat{K}_0$, $\hat{K}_{0.05}$ estimated by Algorithm 1 with standard normal distribution and $N=1,000$. Since the dimension is relatively small, we recommend using a larger significance level $\nu$ to prevent under-specifying. It seems that the variables $\{22,3,	4\}$ are significant and variables $\{21,11\}$ are potentially significant too. KAIC and KCp, however,  select many more variables, which are likely to be spurious. 
\begin{figure}[htbp!]
\center

		\includegraphics[width=12cm,height=8cm]{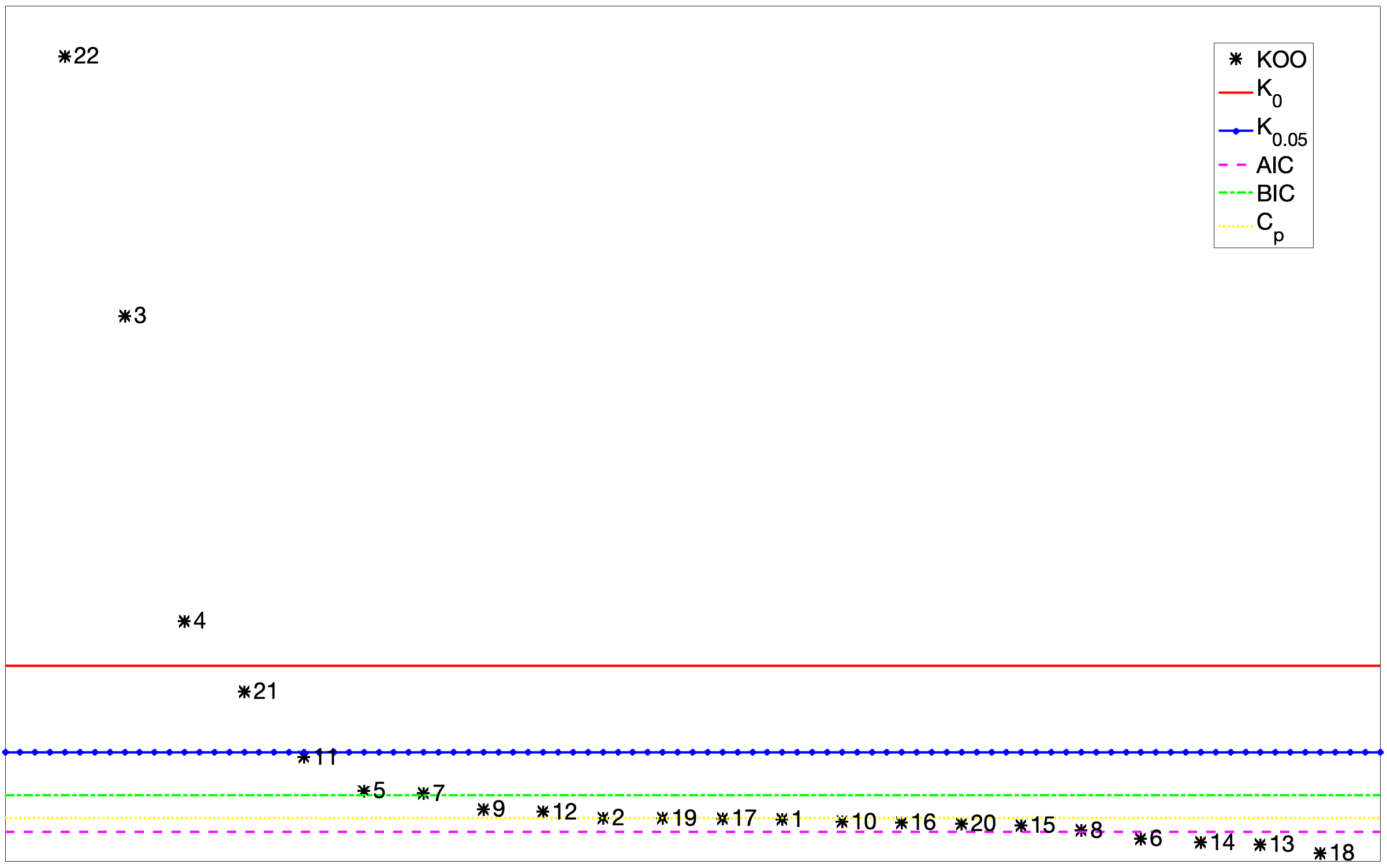}
		
		\caption{Scatterplots for the chemometrics dataset.}
		\label{real_fig1}
\end{figure}
The second example is a multivariate yeast cell-cycle dataset from \cite{SpellmanS98C}, which can be found in the R package ``spls".
This data set contains 542 cell-cycle-related genes (i.e., $n = 542$). Each gene contains 106 binding levels of transcription factors (i.e., $k = 106$) and 18 time points covering two cell cycles (i.e., $p = 18$). The binding levels of the transcription factors play a role in determining which genes are expressed and help delineate the process behind eukaryotic cell cycles. Further explanations of the dataset can be found in \citep{WangC07G, ChunK10S,ChenH12S,KongL17I}. Our results are presented in Figure \ref{real_fig2}. The transcription factors \{SWI5,	 STE12,	 ACE2,	 NDD1\},  corresponding to the four largest $\cK_j$-values, have been confirmed to be related to the cell cycle regulation by experiment  \cite{WangC07G}. And the other two transcription factors \{RME1,	HIR2\} could possibly be related to the cell cycle regulation. 
KBIC, however, will have missed identifying the TFs \{STE12,	 ACE2,	 NDD1\}  in the yeast cell-cycle. On the other hand,  KAIC and KCp will have identified more TFs, many of which may not be related to the yeast cell-cycle.

%
\begin{figure}[htbp!]
\center

		\includegraphics[width=12cm,height=8cm]{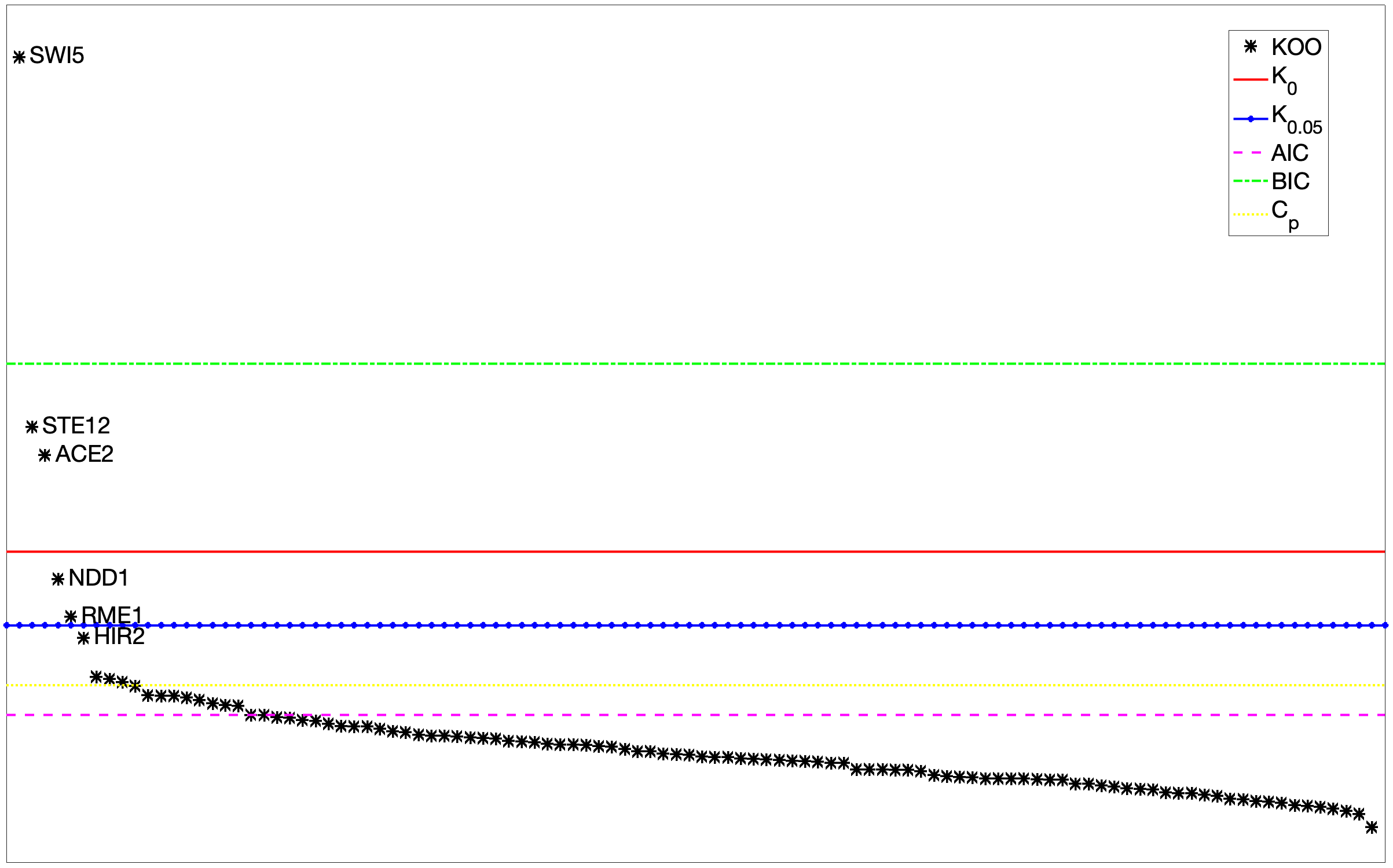}
		
		\caption{Scatterplots for the yeast cell-cycle dataset.}
		\label{real_fig2}
\end{figure}

\section{Proofs of Theorems \ref{limit1}, \ref{clt1} and \ref{clt2}  under normality}\label{proofnorm}
If the errors follow the standard normal distribution, the KOO statistic can be written as the quotient of two independent chi-squared random variables. As a result,  the proofs of Theorems \ref{limit1}, \ref{clt1} and \ref{clt2} are easier to present.  The proofs may also be of independent interest.  Hence,  we prove these results under normality in this section. The proofs of these theorems for general error distributions via random matrix theory are postponed in the Appendix for interested readers.
Note that there is no need to prove Theorem \ref{hattau} when the errors follow the standard normal distribution.  

Recall the KOO statistic
\begin{equation*}
\cK_j=\bbv_j' \bbW^{-1}\bbv_j,
\end{equation*}
where
\begin{equation*}
\bbv_j=\bgS^{-1/2}\bbY'\bba_j~~\mbox{and}~~
\bbW=\bbE'\bbQ\bbE.
\end{equation*}
When $\bbE$ has the standard normal distribution,  it follows  that
$$
\bbW \sim W_{p}(n-k, \bbI_p), ~~~~  \bbv_j \sim N_p(\bgS^{-1/2}\bm\Theta_*\bbX_*'\bba_{j}, \bbI_p), \quad j=1, \ldots, k,   $$ and
$ \bbW$ and $\{\bbv_1, \dots, \bbv_k \}$  are  independent.
Note that $\bbv_1, \dots, \bbv_k$ are not necessarily independent.   If $j \not\in \bbj_*$, $\bgS^{-1/2}\bm\Theta_*\bbX_*'\bba_{j}=\bf0$, on the other hand if $j \in \bbj_*$, $\bgS^{-1/2}\bm\Theta_*\bbX_*'\bba_{j} \not= \bf0$. Moreover, under the assumption of normality, $\tau=0$ in assumption (C3). Next, we state a preliminary lemma.
\begin{lemma} \label{lem3-1}
Let $\bbV=(\bbv_1, \ldots, \bbv_q)$ be a $p \times q$ random matrix with $q \le p$, and  let $\bbW$ be a $p \times p$ random matrix which is distributed as  Wishart distribution $W_p(m, \bbI_p)$. Assume that $\bbV$ and $\bbW$ are independent.  Let $\bbH$ be a $p \times p$ random orthogonal matrix such that the first $q$ columns are $\bbV(\bbV'\bbV)^{-1/2}$, that is, $\bbH=(\bbV(\bbV'\bbV)^{-1/2}, \cdot )$. Let
\begin{equation}
\bbZ=\bbH'\bbW\bbH=\left(
\begin{array}{cc}
\bbZ_{11} & \bbZ_{12}\\
\bbZ_{21} & \bbZ_{22}
\end{array}
\right),
\quad \bbZ_{11\cdot 2}=\bbZ_{11}-\bbZ_{12}\bbZ_{22}^{-1}\bbZ_{21},
\label{decomV}
\end{equation}
and $\bbZ_{21}$ is a $(p-q) \times q$ matrix. Then,  
\begin{align}
&\bbZ \sim W_p(m, \bbI_p),\label{leeq1}\\
&\bbV'\bbW^{-1}\bbV=(\bbV'\bbV)^{1/2}\bbZ_{11\cdot2}^{-1}
(\bbV'\bbV)^{1/2},\label{leeq2}\\
& \bbZ_{11\cdot2} \sim W_q(m-(p-q), \bbI_q).  \label{leeq3}
\end{align}
When  $q=1$,
$$
\bbv_1'\bbW^{-1}\bbv_1=\frac{\bbv_1'\bbv_1}{\bbZ_{11\cdot2}},
$$
where $\bbZ_{11\cdot2} \sim \chi^{2}(m-(p-1))$ is a chi-square variate with
$m-(p-1)$ degrees of freedom,
and $\bbv_1'\bbv_1$ and $\bbZ_{11\cdot2}$ are independent.
Further, if $\bbv_1 \sim N_p(\bm\mu, \bbI_p)$,
$$
\bbv_1'\bbW^{-1}\bbv_1=\frac{\bbv_1'\bbv_1}{\bbZ_{11\cdot2}}\sim\frac{\chi^{2}(p; \bm\mu'\bm\mu)}{\chi^{2}(m-(p-1))}.
$$
Here, $\chi^{2}(p; \delta^2)$ denotes a noncentral chi-square variate with
$p$ degrees of freedom and noncetrality parameter $\bm\mu'\bm\mu$, and
 $\chi^{2}(m-(p-1))$ denotes a chi-square variate with
$m-(p-1)$ degrees of freedom, and they are independent.
\end{lemma}

\begin{proof}
	The result \eqref{leeq1} is straightforward by considering the conditional distribution of
$\bbZ$ given $\bbH$, and noting that the obtained result does not depend on $\bbH$.
Next we consider the result \eqref{leeq2}. Noting that $\bbH$ is orthogonal, we have
\begin{align*}
\bbV'\bbW^{-1}\bbV
&=\bbV'\bbH(\bbH'\bbW\bbH)^{-1}\bbH'\bbV \\
&=\left[(\bbV'\bbV)^{1/2}, \ \bbO \right]'\bbZ^{-1}
\left[(\bbV'\bbV)^{1/2}, \ \bbO \right],
\end{align*}
which implies \eqref{leeq2}. \eqref{leeq3} is a well known result on Wishart distribution (e.g., Theorem 2.2.3 in \citep{FujikoshiU10M}).

Next we consider the case of $q=1$. Note that the $(1,1)$ element of $\bbZ^{-1}$ is $\bbZ_{11\cdot 2}^{-1}$.
Then
\begin{align*}
\bbv_1'\bbW^{-1}\bbv_1=&\bbv_1\bbH'(\bbH'\bbW\bbH)^{-1}\bbH'\bbv_1\\
=& ((\bbv_1'\bbv_1)^{1/2},0, \ldots, 0) \bbZ^{-1}((\bbv_1'\bbv_1)^{1/2},0, \ldots, 0)' \\
=& \bbv_1'\bbv_1\bbZ_{11\cdot 2}^{-1}.
\end{align*}
The required result follows from the fact that $\bbv_1'\bbv_1 \sim \chi^{2}(p ;\delta^2)$ and $\bbZ_{11\cdot 2} \sim \chi^{2}(m-(p-1))$.  Then we complete the proof of this lemma.
\end{proof}

\subsection{Proof of Theorem \ref{limit1}}
By Lemma \ref{lem3-1},
we can express $\cK_j$ as a ratio of two independent chi-square variates as
$$
\cK_j= \left\{
 \begin{array}{lr}
 \frac{\chi^{2}(p)}{\chi^{2}(\widetilde{m})}& \mbox{~~if~~}j\notin\bbj_* \\
\frac{\chi^{2}(p; p\delta_j)}{\chi^{2}(\tilde{m})}& \mbox{~~if~~}j\in\bbj_*
 \end{array}
\right. ,
$$
where $\delta_j=p^{-1}\bbx_{j}'\bbQ_j\bbx_{j}\theta_j'\bgS^{-1}\theta_j $ and $\tilde{m}=n-k-p+1$. 
For  $j \notin \bbj_*$,  let
$$
Z_1=\frac{\chi^2(p)-p}{\sqrt{2p}}~~\mbox{and}~~
Z_2=\frac{\chi^2(\widetilde{m})-\widetilde{m}}{\sqrt{2\widetilde{m}}}.
$$
Then, it is clear that
\begin{align*}
\cK_j&=\frac{p+\sqrt{2p}Z_1}{\widetilde{m}+\sqrt{\widetilde{m}}Z_2}
=\frac{p/n+\sqrt{2p}Z_1/n}{\widetilde{m}/n+\sqrt{\widetilde{m}}Z_2/n}\\
&=\frac{c_n}{1-c_n-\alpha_n} + o_{a.s.}(1).
\end{align*}

For  $j \in \bbj_*$,  let
$$
\widetilde{Z}_1=\frac{\chi^2(p; p\delta_j)-p(1+\delta_j)}
{\sqrt{2p(1+2\delta_j)}}.
$$
Note that $\widetilde{Z}_1\dto N(0,1)$ as $p\to\infty$ or $p\delta_j^2\to\infty$. Thus we can find
 that
\begin{align*}
\cK_j&=\left\{
p(1+\delta_j)+\sqrt{2p(1+2\delta_j)}\widetilde{Z}_1
\right\}
\left\{
\widetilde{m}+\sqrt{2\widetilde{m}}Z_2
\right\}^{-1} \\
&=\frac{p}{\widetilde{m}}
\left\{
1+\delta_j+\sqrt{2(1+2\delta_j)p^{-1}}\widetilde{Z}_1
\right\}
\left\{
1-\sqrt{2\widetilde{m}^{-1}}Z_2
\right\}^{-1},
\end{align*}
which implies
$$
\cK_j - \frac{c_n(1+\delta_{j})}{1-c_n-\alpha_n}=o_{a.s.}(1+\delta_{j}).
$$
Then we complete the proof of Theorem \ref{limit1}.

\subsection{Proof of Theorem \ref{clt1}}
 For simplicity, we  consider the case $q=2$, and
assume that $\{1,2\}\subset [k]\backslash\bbj_*$.
To prove Theorem \ref{clt1}, it is sufficient to show that for any non-null vector $\bbh=(h_1,h_2)'$, $\sqrt{p}[h_1\cK_{1}+h_2\cK_{q}-\frac{c_n}{1-c_n-\alpha_n}(h_1+h_2)]
$ converges weakly to a normal distribution with mean zero and variance $\frac{2c^2(1-\alpha)}{(1-\alpha-c)^3}\bbh'(\cA_2'\cA_2)^2\bbh.$

Under the normality assumption, we can express $\cK_1$ and $\cK_2$ as follows:
\begin{equation}
\cK_1=\bbv_1'\bbW^{-1}\bbv_1, \quad \cK_2=\bbv_2'\bbW^{-1}\bbv_2.
\label{k1k2}
\end{equation}
Here, $\bbv_i \sim N_p({\bf 0}, \bbI_p), i=1, 2$, $\bbW \sim W_p(m, \bbI_p)$,
$m=n-k$,
$\{\bbv_1, \bbv_2\}$ and $\bbW$ are independent, but
 $\bbv_1$ and $\bbv_2$ are not independent. Let
$\cK_0=\bbv_1'\bbW^{-1}\bbv_2$ and
$$
{\bm \cK}=\begin{pmatrix}
\cK_1 & \cK_0 \\
\cK_0' & \cK_2
\end{pmatrix}.
$$
Note that
$$
h_1\cK_1+h_2\cK_2=\tr D_h{\bm \cK},
$$
where $D_h=\begin{pmatrix}
h_1 & 0 \\
0&h_2
\end{pmatrix}$.

Let $\bbV=(\bbv_1, \bbv_2).$
Using Lemma \ref{lem3-1},
we can write ${\bm \cK}$ as
\begin{align}
{\bm \cK} &=\bbV'\bbW^{-1}\bbV
\nonumber \\
&=\left(\frac{1}{p}\bbV'\bbV\right)^{1/2}
\left(\frac{1}{\widetilde{m}}\bbZ_{11\cdot 2}\right)^{-1}
\left(\frac{1}{p}\bbV'\bbV\right)^{1/2}\frac{p}{\widetilde{m}},
\label{kappa}
\end{align}
where $\bbZ_{11\cdot 2} \sim W_2(\tilde{m}, \bbI_p)$ is defined in \eqref{decomV} and $\widetilde{m}=m-(p-2)$.
Note that
$$
\bbV'\bbV \sim W_2(p, {\bm \Lambda}), \quad {\bm \Lambda}=\begin{pmatrix}
1 & \lambda \\
\lambda & 1
\end{pmatrix},
$$
where $\lambda=\bba_1'\bba_2.$
Let
\begin{align}
\bbF&=
\begin{pmatrix}
f_1 & f_3 \\
f_3 & f_2
\end{pmatrix}
=\sqrt{p}\left(\frac{1}{p}\bbV'\bbV - {\bm \Lambda}\right),
\label{f10}
\\
\bbG&=
\begin{pmatrix}
g_1 & g_3 \\
g_3 & g_2
\end{pmatrix}
=\sqrt{\widetilde{m}}\left(\frac{1}{\widetilde{m}}
\bbZ_{11\cdot 2}- \bbI_2\right).
\label{g10}
\end{align}
It follows from the asymptotic distribution of a Wishart matrix (e.g., Theorem 2.5.1 in \citep{FujikoshiU10M}) that
the limiting distribution of $(f_1,f_2,f_3)'$ (respectively,  $(g_1,g_2,g_3)'$) is a $3$-variate normal distribution with mean zero and covariance matrix
$$\begin{pmatrix}
2 & 2\lambda^2&2\lambda \\
2\lambda^2 &2&2\lambda\\
2\lambda&2\lambda&1+\lambda^2
\end{pmatrix}, ~\left(\mbox{respectively, } \begin{pmatrix}
2 & 0&0\\
0&2&0\\
0&0&1
\end{pmatrix}\right).
$$
Consequently, it is straightforward to show that
\begin{align}\label{DhF}
	  \tr D_h\bbF\dto N_2(0,2\tr (D_h\Lambda)^2)
\end{align}
and
\begin{align}\label{DhG}
	  \tr {\bm \Lambda}^{1/2}D_h {\bm \Lambda}^{1/2}\bbG\dto N_2(0,2 \tr ({D_h \bm \Lambda})^2).
\end{align}
Then, by substituting
$$
\frac{1}{p}\bbV'\bbV=\Lambda+ \frac{1}{\sqrt{p}}\bbF
$$
and
$$
\left(\frac{1}{\widetilde{m}} \bbZ_{11\cdot2} \right)^{-1}
=\left(\bbI_2+\frac{1}{\sqrt{\widetilde{m}}}\bbG\right)^{-1}
=\bbI_2-\frac{1}{\sqrt{\widetilde{m}}}\bbG+\frac{1}{{\widetilde{m}}}\left(\bbI_2+\frac{1}{\sqrt{\widetilde{m}}}\bbG\right)^{-1}\bbG^2
$$
into (\ref{kappa}),
we can expand ${\bm \cK}$ as
\begin{equation}
{\bm \cK}=\left\{\Lambda + \frac{1}{\sqrt{p}}\bbF -
\frac{1}{\sqrt{\widetilde{m}}}
(\Lambda+ \frac{1}{\sqrt{p}}\bbF)^{1/2}\bbG(\Lambda+ \frac{1}{\sqrt{p}}\bbF)^{1/2}
+ \O\right\}\frac{c_n}{1-c_n-\alpha_n},
\label{exp1}
\end{equation}
where $\O$ denotes the terms of order $O_p(n^{-1})$.
Using (\ref{exp1}), we have
\begin{align}
&\sqrt{p}\left\{
h_1\cK_1+h_2\cK_2 - \frac{c_n}{1-c_n-\alpha_n}(h_1+h_2)\right\}
\nonumber \\
&=\frac{c_n}{1-c_n-\alpha_n}
\left\{
\tr D_h\bbF -
\left(\frac{c_n}{1-c_n-\alpha_n}\right)^{1/2}
\tr {\bm \Lambda}^{1/2}D_h {\bm \Lambda}^{1/2}\bbG
\right\}+O_p(n^{-1/2}).
\label{exp2}
\end{align}
By \eqref{DhF} and \eqref{DhG}, we can see  that the limiting distribution of \eqref{exp2} is  normal  with mean zero and variance
$$
\frac{2c^2}{(1-c-\alpha)^2}\left(1+\frac{c}{1-c-\alpha}\right)\tr ({D_h \bm \Lambda})^2=\frac{2c^2(1-\alpha)}{(1-\alpha-c)^3}\bbh'(\cA_2'\cA_2)^2\bbh.
$$
This  completes the proof of Theorem \ref{clt1} .

\subsection{Proof of Theorem \ref{clt2}}
In the proof of Theorem \ref{limit1},  recall that for  $j \in \bbj_*$,
 $\cK_j$  can be expressed as a ratio of two independent chi-square variates:
$$
\cK_j=\frac{\chi^{2}(p; p\delta_j^2)}{\chi^{2}(\tilde{m})},
$$
where  $\chi^{2}(p; p\delta_j^2)$ denotes a noncentral chi-square variate with
$p$ degrees of freedom and noncentrality parameter $p\delta_j^2$, and
 $\chi^{2}(\tilde{m})$ denotes a chi-square variate with
$\tilde{m}=n-k-p+1$ degrees of freedom, and they are independent.
Let
$$
\widetilde{Z}_1=\frac{\chi^2(p;p\delta_j^2)-p-p\delta_j^2}
{\sqrt{2(p+2p\delta_j^2)}}, \quad
\widetilde{Z}_2=\frac{\chi^2(\widetilde{m})-\widetilde{m}}{\sqrt{2\widetilde{m}}}.
$$
Then, it is checked that $\widetilde{Z}_1$ and $\widetilde{Z}_2$ converge to the standard normal distribution. Note that
\begin{align*}
\cK_j&=\left\{
(p+p\delta_j^2)+\sqrt{2(p+2p\delta_j^2)}\widetilde{Z}_1
\right\}
\left\{
\widetilde{m}+\sqrt{2\widetilde{m}}\widetilde{Z}_2
\right\}^{-1} \\
&=\frac{p}{\widetilde{m}}
\left\{
1+\delta_j^2+\sqrt{2p^{-1}(1+2\delta_j^2)}\widetilde{Z}_1
\right\}
\left\{
1+\sqrt{2\widetilde{m}^{-1}}\widetilde{Z}_2\right\}^{-1}.
\end{align*}
This implies that
\begin{align*}
&\sqrt{p}\left\{\cK_j - \frac{p}{\widetilde{m}}(1+\delta_j^2)\right\} \\
& \quad \quad
=\frac{p}{\widetilde{m}}\left\{
\sqrt{2(1+2\delta_j^2)}\widetilde{Z}_1-(1+\delta_j^2)\left(\frac{2p}{\widetilde{m}}\right)^{1/2}
\widetilde{Z}_2\right\} + O_p({n^{-1/2}}).
\end{align*}
Theorem \ref{clt2}  follows from noting that $\widetilde{Z}_1$ and $\widetilde{Z}_2$ independently converge to the standard normal distribution.


\section*{Acknowledgement}
The authors are grateful to Professor Xuming He for his helpful suggestions and help in polishing the English.
Bai's research was supported by NSFC No. 12171198 and STDFJ No. 20210101147JC. Choi's research was supported by the Singapore MOE Academic Research Funds R-155-000-222-114. Hu's research was supported by NSFC Nos. 12171078, 11971097, 12292980, 12292982.
 
\appendix

\section{Simulation results}\label{appendixA}
The simulation results for Settings I and II, and six cases of distribution of  $\bbE$ are tabulated in Tables \ref{KOO1-1}--\ref{KOO6}.

 \begin{table}[htbp]
 \center
	\begin{tabular}{|c|ccccc|ccccc|}
		\hline
\multicolumn{11}{|c|}{$\alpha=0.2$, $c=0.2$} \\
\cline{1-11}
		 &\multicolumn{5}{c|}{$n=100$} &\multicolumn{5}{c|}{$n=500$} \\ \cline{2-11}
		& ${\hat\bbj_*^{A}}$& $\hat\bbj_*^{B}$&$\hat\bbj_*^{C}$&$\hat\bbj_*^{(0)}$&$\hat\bbj_*^{(5)}$&${\hat\bbj_*^{A}}$& $\hat\bbj_*^{B}$&$\hat\bbj_*^{C}$&$\hat\bbj_*^{(0)}$&$\hat\bbj_*^{(5)}$\\

\hline
U-S & 0 & 79 & 0 & 15 & 0 & 0 & 589 & 0 & 0 & 0 \\
\hline
T-S & 198 & 921 & 228 & 983 & 966 & 570 & 411 & 655 & 999 & 954 \\
\hline
O-S & 802 & 0 & 772 & 2 & 34 & 430 & 0 & 345 & 1 & 46 \\

A-S & \pmb{2.05} & -- & 1.97 & 1 & 1 & 1.37 & -- & 1.27 & 1 & 1.04 \\
\hline
\cline{1-11}
		 &\multicolumn{5}{c|}{$n=1000$} &\multicolumn{5}{c|}{$n=2000$} \\ \cline{2-11}
		& ${\hat\bbj_*^{A}}$& $\hat\bbj_*^{B}$&$\hat\bbj_*^{C}$&$\hat\bbj_*^{(0)}$&$\hat\bbj_*^{(5)}$&${\hat\bbj_*^{A}}$& $\hat\bbj_*^{B}$&$\hat\bbj_*^{C}$&$\hat\bbj_*^{(0)}$&$\hat\bbj_*^{(5)}$\\
\hline
U-S & 0 & 993 & 0 & 0 & 0 & 0 & 1000 & 0 & 0 & 0 \\
\hline
T-S & 956 & 7 & 972 & 1000 & 958 & 1000 & 0 & 1000 & 999 & 951 \\
\hline
O-S & 44 & 0 & 28 & 0 & 42 & 0 & 0 & 0 & 1 & 49 \\

A-S & 1.02 & -- & 1 & -- & 1.02 & -- & -- & -- & 1 & 1.04 \\
\hline
		\hline
\multicolumn{11}{|c|}{$\alpha=0.2$, $c=0.4$} \\
\cline{1-11}
		 &\multicolumn{5}{c|}{$n=100$} &\multicolumn{5}{c|}{$n=500$} \\ \cline{2-11}
		& ${\hat\bbj_*^{A}}$& $\hat\bbj_*^{B}$&$\hat\bbj_*^{C}$&$\hat\bbj_*^{(0)}$&$\hat\bbj_*^{(5)}$&${\hat\bbj_*^{A}}$& $\hat\bbj_*^{B}$&$\hat\bbj_*^{C}$&$\hat\bbj_*^{(0)}$&$\hat\bbj_*^{(5)}$\\
\hline
U-S & 0 & 938 & 0 & 640 & 19 & 0 & 1000 & 0 & 0 & 0 \\
\hline
T-S & 35 & 62 & 0 & 360 & 940 & 2 & 0 & 0 & 1000 & 953 \\
\hline
O-S & 965 & 0 & 1000 & 0 & 41 & 998 & 0 & 1000 & 0 & 47 \\

A-S & 3.69 & -- & 7.06 & -- & 1.05 & 6.86 & -- & 46.30 & -- & 1.04 \\
\hline
\cline{1-11}
		 &\multicolumn{5}{c|}{$n=1000$} &\multicolumn{5}{c|}{$n=2000$} \\ \cline{2-11}
		& ${\hat\bbj_*^{A}}$& $\hat\bbj_*^{B}$&$\hat\bbj_*^{C}$&$\hat\bbj_*^{(0)}$&$\hat\bbj_*^{(5)}$&${\hat\bbj_*^{A}}$& $\hat\bbj_*^{B}$&$\hat\bbj_*^{C}$&$\hat\bbj_*^{(0)}$&$\hat\bbj_*^{(5)}$\\
\hline

U-S & 0 & 1000 & 0 & 0 & 0 & 0 & 1000 & 0 & 0 & 0 \\
\hline
T-S & 23 & 0 & 0 & 998 & 957 & 505 & 0 & 0 & 1000 & 961 \\
\hline
O-S & 977 & 0 & 1000 & 2 & 43 & 495 & 0 & 1000 & 0 & 39 \\

A-S & 3.92 & -- & 95.28 & 1 & 1.05 & 1.47 & -- & 194.82 & -- & 1 \\
\hline
		\hline
\multicolumn{11}{|c|}{$\alpha=0.4$, $c=0.2$} \\
\cline{1-11}
		 &\multicolumn{5}{c|}{$n=100$} &\multicolumn{5}{c|}{$n=500$} \\ \cline{2-11}
		& ${\hat\bbj_*^{A}}$& $\hat\bbj_*^{B}$&$\hat\bbj_*^{C}$&$\hat\bbj_*^{(0)}$&$\hat\bbj_*^{(5)}$&${\hat\bbj_*^{A}}$& $\hat\bbj_*^{B}$&$\hat\bbj_*^{C}$&$\hat\bbj_*^{(0)}$&$\hat\bbj_*^{(5)}$\\ \hline
U-S & 0 & 42 & 0 & 828 & 41 & 0 & 129 & 0 & 0 & 0 \\
\hline
T-S & 0 & 923 & 3 & 172 & 919 & 0 & 871 & 0 & 998 & 965 \\
\hline
O-S & 1000 & 35 & 997 & 0 & 40 & 1000 & 0 & 1000 & 2 & 35 \\

A-S & 16.50 & 1.09 & 6.67 & -- & 1.12 & 100.87 & -- & 8 & 1 & 1 \\
\hline
\cline{1-11}
		 &\multicolumn{5}{c|}{$n=1000$} &\multicolumn{5}{c|}{$n=2000$} \\ \cline{2-11}
		& ${\hat\bbj_*^{A}}$& $\hat\bbj_*^{B}$&$\hat\bbj_*^{C}$&$\hat\bbj_*^{(0)}$&$\hat\bbj_*^{(5)}$&${\hat\bbj_*^{A}}$& $\hat\bbj_*^{B}$&$\hat\bbj_*^{C}$&$\hat\bbj_*^{(0)}$&$\hat\bbj_*^{(5)}$\\
\hline
U-S & 0 & 729 & 0 & 0 & 0 & 0 & 999 & 0 & 0 & 0 \\
\hline
T-S & 0 & 271 & 41 & 1000 & 954 & 0 & 1 & 748 & 995 & 940 \\
\hline
O-S & 1000 & 0 & 959 & 0 & 46 & 1000 & 0 & 252 & 5 & 60 \\

A-S & 213.52 & -- & 3.28 & -- & 1.02 & 450.55 & -- & 1.16 & 1 & 1 \\
\hline
		\hline
\multicolumn{11}{|c|}{$\alpha=0.4$, $c=0.4$} \\
\cline{1-11}
		 &\multicolumn{5}{c|}{$n=100$} &\multicolumn{5}{c|}{$n=500$} \\ \cline{2-11}
		& ${\hat\bbj_*^{A}}$& $\hat\bbj_*^{B}$&$\hat\bbj_*^{C}$&$\hat\bbj_*^{(0)}$&$\hat\bbj_*^{(5)}$&${\hat\bbj_*^{A}}$& $\hat\bbj_*^{B}$&$\hat\bbj_*^{C}$&$\hat\bbj_*^{(0)}$&$\hat\bbj_*^{(5)}$\\ \hline
U-S & 0 & 623 & 0 & 999 & 889 & 0 & 1000 & 0 & 10 & 0 \\
\hline
T-S & 0 & 294 & 0 & 1 & 103 & 0 & 0 & 0 & 990 & 963 \\
\hline
O-S & 1000 & 83 & 1000 & 0 & 8 & 1000 & 0 & 1000 & 0 & 37 \\

A-S & 31.05 & 1.51 & 29.35 & -- & 1.25 & 194.59 & -- & 193.17 & -- & 1.05 \\
\hline
\cline{1-11}
		 &\multicolumn{5}{c|}{$n=1000$} &\multicolumn{5}{c|}{$n=2000$} \\ \cline{2-11}
		& ${\hat\bbj_*^{A}}$& $\hat\bbj_*^{B}$&$\hat\bbj_*^{C}$&$\hat\bbj_*^{(0)}$&$\hat\bbj_*^{(5)}$&${\hat\bbj_*^{A}}$& $\hat\bbj_*^{B}$&$\hat\bbj_*^{C}$&$\hat\bbj_*^{(0)}$&$\hat\bbj_*^{(5)}$\\
\hline
U-S & 0 & 1000 & 0 & 0 & 0 & 0 & 1000 & 0 & 0 & 0 \\
\hline
T-S & 0 & 0 & 0 & 998 & 952 & 0 & 0 & 0 & 996 & 930 \\
\hline
O-S & 1000 & 0 & 1000 & 2 & 48 & 1000 & 0 & 1000 & 4 & 70 \\

A-S & 394.99 & -- & 394.85 & 1 & 1.06 & 795 & -- & 795 & 1 & 1.01 \\
\hline

	\end{tabular}
	\caption{ Selection times of the KOO methods with AIC, BIC, $C_p$ thresholds and bootstrap methods under Settings  (I)  and (i) based on 1,000 replications. Here U-S , T-S, O-S and A-S stand for number of times a selection method under-specified the true model,  number of times a selection method identified the true model exactly, number of times a selection method over-specified the true model,  and  the average number of spurious variables a selection method identified when it over-specified the model, respectively.
	}\label{KOO1-1}
\end{table}

 \begin{table}[htbp]
 \center
	\begin{tabular}{|c|ccccc|ccccc|}
		\hline
\multicolumn{11}{|c|}{$\alpha=0.2$, $c=0.2$} \\
\cline{1-11}
		 &\multicolumn{5}{c|}{$n=100$} &\multicolumn{5}{c|}{$n=500$} \\ \cline{2-11}
		& ${\hat\bbj_*^{A}}$& $\hat\bbj_*^{B}$&$\hat\bbj_*^{C}$&$\hat\bbj_*^{(0)}$&$\hat\bbj_*^{(5)}$&${\hat\bbj_*^{A}}$& $\hat\bbj_*^{B}$&$\hat\bbj_*^{C}$&$\hat\bbj_*^{(0)}$&$\hat\bbj_*^{(5)}$\\

\hline
U-S & 1 & 109 & 1 & 45 & 10 & 0 & 307 & 0 & 0 & 0 \\
\hline
T-S & 198 & 891 & 219 & 955 & 958 & 563 & 693 & 646 & 1000 & 961 \\
\hline
O-S & 801 & 0 & 780 & 0 & 32 & 437 & 0 & 354 & 0 & 39 \\

A-S & 2.09 & -- & 1.96 & -- & 1.03 & 1.35 & -- & 1.25 & -- & 1.05 \\
\hline
\cline{1-11}
		 &\multicolumn{5}{c|}{$n=1000$} &\multicolumn{5}{c|}{$n=2000$} \\ \cline{2-11}
		& ${\hat\bbj_*^{A}}$& $\hat\bbj_*^{B}$&$\hat\bbj_*^{C}$&$\hat\bbj_*^{(0)}$&$\hat\bbj_*^{(5)}$&${\hat\bbj_*^{A}}$& $\hat\bbj_*^{B}$&$\hat\bbj_*^{C}$&$\hat\bbj_*^{(0)}$&$\hat\bbj_*^{(5)}$\\
\hline
U-S & 0 & 780 & 0 & 0 & 0 & 0 & 1000 & 0 & 0 & 0 \\
\hline
T-S & 963 & 220 & 978 & 999 & 946 & 1000 & 0 & 1000 & 999 & 953 \\
\hline
O-S & 37 & 0 & 22 & 1 & 54 & 0 & 0 & 0 & 1 & 47 \\

A-S & 1.03 & -- & 1 & 1 & 1.02 & -- & -- & -- & 1 & 1.04 \\
\hline
		\hline
\multicolumn{11}{|c|}{$\alpha=0.2$, $c=0.4$} \\
\cline{1-11}
		 &\multicolumn{5}{c|}{$n=100$} &\multicolumn{5}{c|}{$n=500$} \\ \cline{2-11}
		& ${\hat\bbj_*^{A}}$& $\hat\bbj_*^{B}$&$\hat\bbj_*^{C}$&$\hat\bbj_*^{(0)}$&$\hat\bbj_*^{(5)}$&${\hat\bbj_*^{A}}$& $\hat\bbj_*^{B}$&$\hat\bbj_*^{C}$&$\hat\bbj_*^{(0)}$&$\hat\bbj_*^{(5)}$\\
\hline
U-S & 2 & 734 & 0 & 487 & 41 & 0 & 1000 & 0 & 0 & 0 \\
\hline
T-S & 40 & 266 & 0 & 513 & 937 & 1 & 0 & 0 & 1000 & 945 \\
\hline
O-S & 958 & 0 & 1000 & 0 & 22 & 999 & 0 & 1000 & 0 & 55 \\

A-S & 3.63 & -- & 6.99 & -- & 1.05 & 6.59 & -- & 45.85 & -- & 1.02 \\
\hline
\cline{1-11}
		 &\multicolumn{5}{c|}{$n=1000$} &\multicolumn{5}{c|}{$n=2000$} \\ \cline{2-11}
		& ${\hat\bbj_*^{A}}$& $\hat\bbj_*^{B}$&$\hat\bbj_*^{C}$&$\hat\bbj_*^{(0)}$&$\hat\bbj_*^{(5)}$&${\hat\bbj_*^{A}}$& $\hat\bbj_*^{B}$&$\hat\bbj_*^{C}$&$\hat\bbj_*^{(0)}$&$\hat\bbj_*^{(5)}$\\
\hline

U-S & 0 & 1000 & 0 & 0 & 0 & 0 & 1000 & 0 & 0 & 0 \\
\hline
T-S & 26 & 0 & 0 & 1000 & 962 & 490 & 0 & 0 & 999 & 943 \\
\hline
O-S & 974 & 0 & 1000 & 0 & 38 & 510 & 0 & 1000 & 1 & 57 \\

A-S & 3.95 & -- & 95.28 & -- & 1 & 1.46 & -- & 194.80 & 1 & 1.02 \\
\hline
		\hline
\multicolumn{11}{|c|}{$\alpha=0.4$, $c=0.2$} \\
\cline{1-11}
		 &\multicolumn{5}{c|}{$n=100$} &\multicolumn{5}{c|}{$n=500$} \\ \cline{2-11}
		& ${\hat\bbj_*^{A}}$& $\hat\bbj_*^{B}$&$\hat\bbj_*^{C}$&$\hat\bbj_*^{(0)}$&$\hat\bbj_*^{(5)}$&${\hat\bbj_*^{A}}$& $\hat\bbj_*^{B}$&$\hat\bbj_*^{C}$&$\hat\bbj_*^{(0)}$&$\hat\bbj_*^{(5)}$\\ \hline
U-S & 0 & 56 & 2 & 245 & 53 & 0 & 113 & 0 & 0 & 0 \\
\hline
T-S & 0 & 917 & 2 & 755 & 913 & 0 & 887 & 0 & 999 & 958 \\
\hline
O-S & 1000 & 27 & 996 & 0 & 34 & 1000 & 0 & 1000 & 1 & 42 \\

A-S & 16.56 & 1.04 & 6.79 & -- & 1.03 & 102.01 & -- & 8.19 & 1 & 1 \\
\hline
\cline{1-11}
		 &\multicolumn{5}{c|}{$n=1000$} &\multicolumn{5}{c|}{$n=2000$} \\ \cline{2-11}
		& ${\hat\bbj_*^{A}}$& $\hat\bbj_*^{B}$&$\hat\bbj_*^{C}$&$\hat\bbj_*^{(0)}$&$\hat\bbj_*^{(5)}$&${\hat\bbj_*^{A}}$& $\hat\bbj_*^{B}$&$\hat\bbj_*^{C}$&$\hat\bbj_*^{(0)}$&$\hat\bbj_*^{(5)}$\\
\hline
U-S & 0 & 465 & 0 & 0 & 0 & 0 & 902 & 0 & 0 & 0 \\
\hline
T-S & 0 & 535 & 54 & 1000 & 940 & 0 & 98 & 780 & 999 & 959 \\
\hline
O-S & 1000 & 0 & 946 & 0 & 60 & 1000 & 0 & 220 & 1 & 41 \\

A-S & 213.34 & -- & 3.27 & -- & 1.02 & 449.91 & -- & 1.11 & 1 & 1.05 \\
\hline
		\hline
\multicolumn{11}{|c|}{$\alpha=0.4$, $c=0.4$} \\
\cline{1-11}
		 &\multicolumn{5}{c|}{$n=100$} &\multicolumn{5}{c|}{$n=500$} \\ \cline{2-11}
		& ${\hat\bbj_*^{A}}$& $\hat\bbj_*^{B}$&$\hat\bbj_*^{C}$&$\hat\bbj_*^{(0)}$&$\hat\bbj_*^{(5)}$&${\hat\bbj_*^{A}}$& $\hat\bbj_*^{B}$&$\hat\bbj_*^{C}$&$\hat\bbj_*^{(0)}$&$\hat\bbj_*^{(5)}$\\ \hline
U-S & 0 & 465 & 1 & 984 & 744 & 0 & 996 & 0 & 17 & 3 \\
\hline
T-S & 0 & 420 & 0 & 16 & 240 & 0 & 4 & 0 & 983 & 937 \\
\hline
O-S & 1000 & 115 & 999 & 0 & 16 & 1000 & 0 & 1000 & 0 & 60 \\

A-S & 30.98 & 1.39 & 29.18 & -- & 1.06 & 194.59 & -- & 193.12 & -- & 1 \\
\hline
\cline{1-11}
		 &\multicolumn{5}{c|}{$n=1000$} &\multicolumn{5}{c|}{$n=2000$} \\ \cline{2-11}
		& ${\hat\bbj_*^{A}}$& $\hat\bbj_*^{B}$&$\hat\bbj_*^{C}$&$\hat\bbj_*^{(0)}$&$\hat\bbj_*^{(5)}$&${\hat\bbj_*^{A}}$& $\hat\bbj_*^{B}$&$\hat\bbj_*^{C}$&$\hat\bbj_*^{(0)}$&$\hat\bbj_*^{(5)}$\\
\hline
U-S & 0 & 1000 & 0 & 3 & 1 & 0 & 1000 & 0 & 0 & 0 \\
\hline
T-S & 0 & 0 & 0 & 996 & 949 & 0 & 0 & 0 & 999 & 950 \\
\hline
O-S & 1000 & 0 & 1000 & 1 & 50 & 1000 & 0 & 1000 & 1 & 50 \\

A-S & 394.99 & -- & 394.86 & 1 & 1 & 795 & -- & 795 & 1 & 1 \\
\hline

	\end{tabular}
	\caption{ Selection times of the KOO methods with AIC, BIC, $C_p$ thresholds and bootstrap methods under Settings  (I)  and (ii) based on 1,000 replications. Here U-S , T-S, O-S and A-S stand for number of times a selection method under-specified the true model,  number of times a selection method identified the true model exactly, number of times a selection method over-specified the true model,  and  the average number of spurious variables a selection method identified when it over-specified the model, respectively.
	}\label{KOO2}
\end{table}

 \begin{table}[htbp] \center
	\begin{tabular}{|c|ccccc|ccccc|}
		\hline
\multicolumn{11}{|c|}{$\alpha=0.2$, $c=0.2$} \\
\cline{1-11}
		 &\multicolumn{5}{c|}{$n=100$} &\multicolumn{5}{c|}{$n=500$} \\ \cline{2-11}
		& ${\hat\bbj_*^{A}}$& $\hat\bbj_*^{B}$&$\hat\bbj_*^{C}$&$\hat\bbj_*^{(0)}$&$\hat\bbj_*^{(5)}$&${\hat\bbj_*^{A}}$& $\hat\bbj_*^{B}$&$\hat\bbj_*^{C}$&$\hat\bbj_*^{(0)}$&$\hat\bbj_*^{(5)}$\\

\hline
U-S & 0 & 120 & 0 & 25 & 0 & 0 & 543 & 0 & 0 & 0 \\
\hline
T-S & 182 & 880 & 204 & 974 & 969 & 572 & 457 & 663 & 999 & 955 \\
\hline
O-S & 818 & 0 & 796 & 1 & 31 & 428 & 0 & 337 & 1 & 45 \\

A-S & 2.10 & -- & 2.02 & 1 & 1.10 & 1.35 & -- & 1.23 & 1 & 1.04 \\
\hline

\cline{1-11}
		 &\multicolumn{5}{c|}{$n=1000$} &\multicolumn{5}{c|}{$n=2000$} \\ \cline{2-11}
		& ${\hat\bbj_*^{A}}$& $\hat\bbj_*^{B}$&$\hat\bbj_*^{C}$&$\hat\bbj_*^{(0)}$&$\hat\bbj_*^{(5)}$&${\hat\bbj_*^{A}}$& $\hat\bbj_*^{B}$&$\hat\bbj_*^{C}$&$\hat\bbj_*^{(0)}$&$\hat\bbj_*^{(5)}$\\
\hline
U-S & 0 & 975 & 0 & 0 & 0 & 0 & 1000 & 0 & 0 & 0 \\
\hline
T-S & 963 & 25 & 981 & 1000 & 963 & 999 & 0 & 1000 & 999 & 953 \\
\hline
O-S & 37 & 0 & 19 & 0 & 37 & 1 & 0 & 0 & 1 & 47 \\

A-S & 1 & -- & 1 & -- & 1 & 1 & -- & -- & 1 & 1.02 \\
\hline
\hline
\multicolumn{11}{|c|}{$\alpha=0.2$, $c=0.4$} \\
\cline{1-11}
		 &\multicolumn{5}{c|}{$n=100$} &\multicolumn{5}{c|}{$n=500$} \\ \cline{2-11}
		& ${\hat\bbj_*^{A}}$& $\hat\bbj_*^{B}$&$\hat\bbj_*^{C}$&$\hat\bbj_*^{(0)}$&$\hat\bbj_*^{(5)}$&${\hat\bbj_*^{A}}$& $\hat\bbj_*^{B}$&$\hat\bbj_*^{C}$&$\hat\bbj_*^{(0)}$&$\hat\bbj_*^{(5)}$\\
		\hline
U-S & 0 & 895 & 0 & 219 & 24 & 0 & 1000 & 0 & 0 & 0 \\
\hline
T-S & 37 & 105 & 1 & 778 & 932 & 4 & 0 & 0 & 999 & 941 \\
\hline
O-S & 963 & 0 & 999 & 3 & 44 & 996 & 0 & 1000 & 1 & 59 \\

A-S & 3.68 & -- & 7.11 & 1 & 1.07 & 6.60 & -- & 46.14 & 1 & 1.02 \\
\hline
\cline{1-11}
		 &\multicolumn{5}{c|}{$n=1000$} &\multicolumn{5}{c|}{$n=2000$} \\ \cline{2-11}
		& ${\hat\bbj_*^{A}}$& $\hat\bbj_*^{B}$&$\hat\bbj_*^{C}$&$\hat\bbj_*^{(0)}$&$\hat\bbj_*^{(5)}$&${\hat\bbj_*^{A}}$& $\hat\bbj_*^{B}$&$\hat\bbj_*^{C}$&$\hat\bbj_*^{(0)}$&$\hat\bbj_*^{(5)}$\\
\hline
U-S & 0 & 1000 & 0 & 0 & 0 & 0 & 1000 & 0 & 0 & 0 \\
\hline
T-S & 35 & 0 & 0 & 1000 & 937 & 454 & 0 & 0 & 1000 & 941 \\
\hline
O-S & 965 & 0 & 1000 & 0 & 63 & 546 & 0 & 1000 & 0 & 59 \\

A-S & 3.90 & -- & 95.25 & -- & 1.05 & 1.46 & -- & 194.82 & -- & 1 \\  \hline
\hline
\multicolumn{11}{|c|}{$\alpha=0.4$, $c=0.2$} \\
\cline{1-11}
		 &\multicolumn{5}{c|}{$n=100$} &\multicolumn{5}{c|}{$n=500$} \\ \cline{2-11}
		& ${\hat\bbj_*^{A}}$& $\hat\bbj_*^{B}$&$\hat\bbj_*^{C}$&$\hat\bbj_*^{(0)}$&$\hat\bbj_*^{(5)}$&${\hat\bbj_*^{A}}$& $\hat\bbj_*^{B}$&$\hat\bbj_*^{C}$&$\hat\bbj_*^{(0)}$&$\hat\bbj_*^{(5)}$\\
		\hline
U-S & 0 & 61 & 0 & 331 & 43 & 0 & 153 & 0 & 0 & 0 \\
\hline
T-S & 0 & 894 & 5 & 666 & 897 & 0 & 847 & 0 & 999 & 969 \\
\hline
O-S & 1000 & 45 & 995 & 3 & 60 & 1000 & 0 & 1000 & 1 & 31 \\

A-S & 16.59 & 1.02 & 6.91 & 1 & 1.08 & 101.70 & -- & 8.23 & 1 & 1 \\

 \hline
\cline{1-11}
		 &\multicolumn{5}{c|}{$n=1000$} &\multicolumn{5}{c|}{$n=2000$} \\ \cline{2-11}
		& ${\hat\bbj_*^{A}}$& $\hat\bbj_*^{B}$&$\hat\bbj_*^{C}$&$\hat\bbj_*^{(0)}$&$\hat\bbj_*^{(5)}$&${\hat\bbj_*^{A}}$& $\hat\bbj_*^{B}$&$\hat\bbj_*^{C}$&$\hat\bbj_*^{(0)}$&$\hat\bbj_*^{(5)}$\\
\hline
U-S & 0 & 715 & 0 & 0 & 0 & 0 & 1000 & 0 & 0 & 0 \\
\hline
T-S & 0 & 285 & 44 & 1000 & 947 & 0 & 0 & 788 & 999 & 963 \\
\hline
O-S & 1000 & 0 & 956 & 0 & 53 & 1000 & 0 & 212 & 1 & 37 \\

A-S & 213.18 & -- & 3.22 & -- & 1.04 & 449.63 & -- & 1.13 & 1 & 1 \\
\hline
\hline
\multicolumn{11}{|c|}{$\alpha=0.4$, $c=0.4$} \\
\cline{1-11}
		 &\multicolumn{5}{c|}{$n=100$} &\multicolumn{5}{c|}{$n=500$} \\ \cline{2-11}
		& ${\hat\bbj_*^{A}}$& $\hat\bbj_*^{B}$&$\hat\bbj_*^{C}$&$\hat\bbj_*^{(0)}$&$\hat\bbj_*^{(5)}$&${\hat\bbj_*^{A}}$& $\hat\bbj_*^{B}$&$\hat\bbj_*^{C}$&$\hat\bbj_*^{(0)}$&$\hat\bbj_*^{(5)}$\\
		\hline
U-S & 0 & 589 & 0 & 994 & 837 & 0 & 1000 & 0 & 2 & 0 \\
\hline
T-S & 0 & 319 & 0 & 6 & 150 & 0 & 0 & 0 & 995 & 950 \\
\hline
O-S & 1000 & 92 & 1000 & 0 & 13 & 1000 & 0 & 1000 & 3 & 50 \\

A-S & 30.97 & 1.34 & 29.24 & -- & 1.38 & 194.60 & -- & 193.15 & 1 & 1.02 \\    \hline
\cline{1-11}
		 &\multicolumn{5}{c|}{$n=1000$} &\multicolumn{5}{c|}{$n=2000$} \\ \cline{2-11}
		& ${\hat\bbj_*^{A}}$& $\hat\bbj_*^{B}$&$\hat\bbj_*^{C}$&$\hat\bbj_*^{(0)}$&$\hat\bbj_*^{(5)}$&${\hat\bbj_*^{A}}$& $\hat\bbj_*^{B}$&$\hat\bbj_*^{C}$&$\hat\bbj_*^{(0)}$&$\hat\bbj_*^{(5)}$\\
\hline
U-S & 0 & 1000 & 0 & 0 & 0 & 0 & 1000 & 0 & 0 & 0 \\
\hline
T-S & 0 & 0 & 0 & 1000 & 950 & 0 & 0 & 0 & 994 & 945 \\
\hline
O-S & 1000 & 0 & 1000 & 0 & 50 & 1000 & 0 & 1000 & 6 & 55 \\

A-S & 394.99 & -- & 394.83 & -- & 1.06 & 795 & -- & 795 & 1 & 1.05 \\
 \hline

	\end{tabular}
	\caption{ Selection times of the KOO methods with AIC, BIC, $C_p$ thresholds and bootstrap methods under Settings (I) and (iii) based on 1,000 replications. Here U-S , T-S, O-S and A-S stand for number of times a selection method under-specified the true model,  number of times a selection method identified the true model exactly, number of times a selection method over-specified the true model,  and  the average number of spurious variables a selection method identified when it over-specified the model, respectively.
	}\label{KOO3}
\end{table}


 \begin{table}[htbp] \center
	\begin{tabular}{|c|ccccc|ccccc|}
		\hline
\multicolumn{11}{|c|}{$\alpha=0.2$, $c=0.2$} \\
\cline{1-11}
		 &\multicolumn{5}{c|}{$n=100$} &\multicolumn{5}{c|}{$n=500$} \\ \cline{2-11}
		& ${\hat\bbj_*^{A}}$& $\hat\bbj_*^{B}$&$\hat\bbj_*^{C}$&$\hat\bbj_*^{(0)}$&$\hat\bbj_*^{(5)}$&${\hat\bbj_*^{A}}$& $\hat\bbj_*^{B}$&$\hat\bbj_*^{C}$&$\hat\bbj_*^{(0)}$&$\hat\bbj_*^{(5)}$\\

\hline
U-S & 0 & 77 & 0 & 971 & 215 & 0 & 648 & 0 & 0 & 0 \\
\hline
T-S & 41 & 847 & 46 & 29 & 753 & 0 & 352 & 1 & 1000 & 967 \\
\hline
O-S & 959 & 76 & 954 & 0 & 32 & 1000 & 0 & 999 & 0 & 33 \\

A-S & 3.11 & 1.04 & 3.01 & -- & 1 & 7.40 & -- & 6.68 & -- & 1.03 \\
\hline

\cline{1-11}
		 &\multicolumn{5}{c|}{$n=1000$} &\multicolumn{5}{c|}{$n=2000$} \\ \cline{2-11}
		& ${\hat\bbj_*^{A}}$& $\hat\bbj_*^{B}$&$\hat\bbj_*^{C}$&$\hat\bbj_*^{(0)}$&$\hat\bbj_*^{(5)}$&${\hat\bbj_*^{A}}$& $\hat\bbj_*^{B}$&$\hat\bbj_*^{C}$&$\hat\bbj_*^{(0)}$&$\hat\bbj_*^{(5)}$\\
\hline
U-S & 0 & 993 & 0 & 0 & 0 & 0 & 1000 & 0 & 0 & 0 \\
\hline
T-S & 4 & 7 & 6 & 1000 & 944 & 177 & 0 & 267 & 997 & 942 \\
\hline
O-S & 996 & 0 & 994 & 0 & 56 & 823 & 0 & 733 & 3 & 58 \\

A-S & 5.51 & -- & 4.65 & -- & 1.04 & 2.11 & -- & 1.77 & 1 & 1.02 \\
\hline
\hline
\multicolumn{11}{|c|}{$\alpha=0.2$, $c=0.4$} \\
\cline{1-11}
		 &\multicolumn{5}{c|}{$n=100$} &\multicolumn{5}{c|}{$n=500$} \\ \cline{2-11}
		& ${\hat\bbj_*^{A}}$& $\hat\bbj_*^{B}$&$\hat\bbj_*^{C}$&$\hat\bbj_*^{(0)}$&$\hat\bbj_*^{(5)}$&${\hat\bbj_*^{A}}$& $\hat\bbj_*^{B}$&$\hat\bbj_*^{C}$&$\hat\bbj_*^{(0)}$&$\hat\bbj_*^{(5)}$\\
		\hline
U-S & 0 & 925 & 0 & 991 & 622 & 0 & 1000 & 0 & 2 & 0 \\
\hline
T-S & 2 & 74 & 0 & 9 & 361 & 0 & 0 & 0 & 997 & 934 \\
\hline
O-S & 998 & 1 & 1000 & 0 & 17 & 1000 & 0 & 1000 & 1 & 66 \\

A-S & 4.74 & 1 & 7 & -- & 1.06 & 16.40 & -- & 45.88 & 1 & 1 \\
\hline
\cline{1-11}
		 &\multicolumn{5}{c|}{$n=1000$} &\multicolumn{5}{c|}{$n=2000$} \\ \cline{2-11}
		& ${\hat\bbj_*^{A}}$& $\hat\bbj_*^{B}$&$\hat\bbj_*^{C}$&$\hat\bbj_*^{(0)}$&$\hat\bbj_*^{(5)}$&${\hat\bbj_*^{A}}$& $\hat\bbj_*^{B}$&$\hat\bbj_*^{C}$&$\hat\bbj_*^{(0)}$&$\hat\bbj_*^{(5)}$\\
\hline
U-S & 0 & 1000 & 0 & 0 & 0 & 0 & 1000 & 0 & 0 & 0 \\
\hline
T-S & 0 & 0 & 0 & 1000 & 938 & 0 & 0 & 0 & 999 & 924 \\
\hline
O-S & 1000 & 0 & 1000 & 0 & 62 & 1000 & 0 & 1000 & 1 & 76 \\

A-S & 18.74 & -- & 95.36 & -- & 1.03 & 13.77 & -- & 194.65 & 1 & 1.01 \\
 \hline
\hline
\multicolumn{11}{|c|}{$\alpha=0.4$, $c=0.2$} \\
\cline{1-11}
		 &\multicolumn{5}{c|}{$n=100$} &\multicolumn{5}{c|}{$n=500$} \\ \cline{2-11}
		& ${\hat\bbj_*^{A}}$& $\hat\bbj_*^{B}$&$\hat\bbj_*^{C}$&$\hat\bbj_*^{(0)}$&$\hat\bbj_*^{(5)}$&${\hat\bbj_*^{A}}$& $\hat\bbj_*^{B}$&$\hat\bbj_*^{C}$&$\hat\bbj_*^{(0)}$&$\hat\bbj_*^{(5)}$\\
		\hline
U-S & 0 & 4 & 0 & 999 & 597 & 0 & 7 & 0 & 7 & 0 \\
\hline
T-S & 0 & 348 & 0 & 1 & 386 & 0 & 993 & 0 & 993 & 961 \\
\hline
O-S & 1000 & 648 & 1000 & 0 & 17 & 1000 & 0 & 1000 & 0 & 39 \\

A-S & 15.31 & 1.67 & 9.23 & -- & 1 & 94.64 & -- & 28.61 & -- & 1 \\
  \hline
\cline{1-11}
		 &\multicolumn{5}{c|}{$n=1000$} &\multicolumn{5}{c|}{$n=2000$} \\ \cline{2-11}
		& ${\hat\bbj_*^{A}}$& $\hat\bbj_*^{B}$&$\hat\bbj_*^{C}$&$\hat\bbj_*^{(0)}$&$\hat\bbj_*^{(5)}$&${\hat\bbj_*^{A}}$& $\hat\bbj_*^{B}$&$\hat\bbj_*^{C}$&$\hat\bbj_*^{(0)}$&$\hat\bbj_*^{(5)}$\\
\hline
U-S & 0 & 27 & 0 & 0 & 0 & 0 & 354 & 0 & 0 & 0 \\
\hline
T-S & 0 & 973 & 0 & 1000 & 939 & 0 & 646 & 0 & 1000 & 956 \\
\hline
O-S & 1000 & 0 & 1000 & 0 & 61 & 1000 & 0 & 1000 & 0 & 44 \\

A-S & 198.92 & -- & 31.12 & -- & 1.03 & 417.94 & -- & 21.32 & -- & 1.02 \\
\hline
\hline
\multicolumn{11}{|c|}{$\alpha=0.4$, $c=0.4$} \\
\cline{1-11}
		 &\multicolumn{5}{c|}{$n=100$} &\multicolumn{5}{c|}{$n=500$} \\ \cline{2-11}
		& ${\hat\bbj_*^{A}}$& $\hat\bbj_*^{B}$&$\hat\bbj_*^{C}$&$\hat\bbj_*^{(0)}$&$\hat\bbj_*^{(5)}$&${\hat\bbj_*^{A}}$& $\hat\bbj_*^{B}$&$\hat\bbj_*^{C}$&$\hat\bbj_*^{(0)}$&$\hat\bbj_*^{(5)}$\\
		\hline
U-S & 0 & 243 & 0 & 1000 & 898 & 0 & 988 & 0 & 188 & 2 \\
\hline
T-S & 0 & 233 & 0 & 0 & 97 & 0 & 12 & 0 & 812 & 965 \\
\hline
O-S & 1000 & 524 & 1000 & 0 & 5 & 1000 & 0 & 1000 & 0 & 33 \\

A-S & 28.58 & 1.99 & 26.89 & -- & 1.20 & 191.27 & -- & 186.32 & -- & 1.03 \\
      \hline
\cline{1-11}
		 &\multicolumn{5}{c|}{$n=1000$} &\multicolumn{5}{c|}{$n=2000$} \\ \cline{2-11}
		& ${\hat\bbj_*^{A}}$& $\hat\bbj_*^{B}$&$\hat\bbj_*^{C}$&$\hat\bbj_*^{(0)}$&$\hat\bbj_*^{(5)}$&${\hat\bbj_*^{A}}$& $\hat\bbj_*^{B}$&$\hat\bbj_*^{C}$&$\hat\bbj_*^{(0)}$&$\hat\bbj_*^{(5)}$\\
\hline
U-S & 0 & 1000 & 0 & 0 & 0 & 0 & 1000 & 0 & 0 & 0 \\
\hline
T-S & 0 & 0 & 0 & 1000 & 942 & 0 & 0 & 0 & 1000 & 928 \\
\hline
O-S & 1000 & 0 & 1000 & 0 & 58 & 1000 & 0 & 1000 & 0 & 72 \\

A-S & 394.33 & -- & 391.74 & -- & 1.05 & 794.99 & -- & 794.74 & -- & 1.04 \\
           \hline

	\end{tabular}
	\caption{ Selection times of the KOO methods with AIC, BIC, $C_p$ thresholds and bootstrap methods under Settings (II) and (iv) based on 1,000 replications. Here U-S, T-S, O-S and A-S stand for number of times a selection method under-specified the true model,  number of times a selection method identified the true model exactly, number of times a selection method over-specified the true model,  and  the average number of spurious variables a selection method identified when it over-specified the model, respectively.
	}\label{KOO4-1}
\end{table}


 \begin{table}[htbp]
 \center
	\begin{tabular}{|c|ccccc|ccccc|}
		\hline
\multicolumn{11}{|c|}{$\alpha=0.2$, $c=0.2$} \\
\cline{1-11}
		 &\multicolumn{5}{c|}{$n=100$} &\multicolumn{5}{c|}{$n=500$} \\ \cline{2-11}
		& ${\hat\bbj_*^{A}}$& $\hat\bbj_*^{B}$&$\hat\bbj_*^{C}$&$\hat\bbj_*^{(0)}$&$\hat\bbj_*^{(5)}$&${\hat\bbj_*^{A}}$& $\hat\bbj_*^{B}$&$\hat\bbj_*^{C}$&$\hat\bbj_*^{(0)}$&$\hat\bbj_*^{(5)}$\\

\hline
U-S & 0 & 63 & 0 & 66 & 0 & 0 & 732 & 0 & 0 & 0 \\
\hline
T-S & 127 & 936 & 144 & 933 & 963 & 169 & 268 & 245 & 998 & 969 \\
\hline
O-S & 873 & 1 & 856 & 1 & 37 & 831 & 0 & 755 & 2 & 31 \\

A-S & 2.38 & 1 & 2.27 & 1 & 1 & 2.16 & -- & 1.92 & 1 & 1 \\
\hline
\cline{1-11}
		 &\multicolumn{5}{c|}{$n=1000$} &\multicolumn{5}{c|}{$n=2000$} \\ \cline{2-11}
		& ${\hat\bbj_*^{A}}$& $\hat\bbj_*^{B}$&$\hat\bbj_*^{C}$&$\hat\bbj_*^{(0)}$&$\hat\bbj_*^{(5)}$&${\hat\bbj_*^{A}}$& $\hat\bbj_*^{B}$&$\hat\bbj_*^{C}$&$\hat\bbj_*^{(0)}$&$\hat\bbj_*^{(5)}$\\
\hline
U-S & 0 & 1000 & 0 & 0 & 0 & 0 & 1000 & 0 & 0 & 0 \\
\hline
T-S & 688 & 0 & 779 & 994 & 949 & 994 & 0 & 998 & 1000 & 937 \\
\hline
O-S & 312 & 0 & 221 & 6 & 51 & 6 & 0 & 2 & 0 & 63 \\

A-S & 1.16 & -- & 1.11 & 1 & 1.02 & 1 & -- & 1 & -- & 1.05 \\
\hline
		\hline
\multicolumn{11}{|c|}{$\alpha=0.2$, $c=0.4$} \\
\cline{1-11}
		 &\multicolumn{5}{c|}{$n=100$} &\multicolumn{5}{c|}{$n=500$} \\ \cline{2-11}
		& ${\hat\bbj_*^{A}}$& $\hat\bbj_*^{B}$&$\hat\bbj_*^{C}$&$\hat\bbj_*^{(0)}$&$\hat\bbj_*^{(5)}$&${\hat\bbj_*^{A}}$& $\hat\bbj_*^{B}$&$\hat\bbj_*^{C}$&$\hat\bbj_*^{(0)}$&$\hat\bbj_*^{(5)}$\\
\hline
U-S & 0 & 973 & 0 & 615 & 100 & 0 & 1000 & 0 & 0 & 0 \\
\hline
T-S & 21 & 27 & 2 & 385 & 863 & 3 & 0 & 0 & 999 & 948 \\
\hline
O-S & 979 & 0 & 998 & 0 & 37 & 997 & 0 & 1000 & 1 & 52 \\

A-S & 3.86 & -- & 6.98 & -- & 1.05 & 8.99 & -- & 46.12 & 1 & 1.02 \\
\hline
\cline{1-11}
		 &\multicolumn{5}{c|}{$n=1000$} &\multicolumn{5}{c|}{$n=2000$} \\ \cline{2-11}
		& ${\hat\bbj_*^{A}}$& $\hat\bbj_*^{B}$&$\hat\bbj_*^{C}$&$\hat\bbj_*^{(0)}$&$\hat\bbj_*^{(5)}$&${\hat\bbj_*^{A}}$& $\hat\bbj_*^{B}$&$\hat\bbj_*^{C}$&$\hat\bbj_*^{(0)}$&$\hat\bbj_*^{(5)}$\\
\hline

U-S & 0 & 1000 & 0 & 0 & 0 & 0 & 1000 & 0 & 0 & 0 \\
\hline
T-S & 2 & 0 & 0 & 999 & 952 & 132 & 0 & 0 & 998 & 947 \\
\hline
O-S & 998 & 0 & 1000 & 1 & 48 & 868 & 0 & 1000 & 2 & 53 \\

A-S & 6.64 & -- & 95.24 & 1 & 1.02 & 2.30 & -- & 193.88 & 1 & 1.06 \\
\hline
		\hline
\multicolumn{11}{|c|}{$\alpha=0.4$, $c=0.2$} \\
\cline{1-11}
		 &\multicolumn{5}{c|}{$n=100$} &\multicolumn{5}{c|}{$n=500$} \\ \cline{2-11}
		& ${\hat\bbj_*^{A}}$& $\hat\bbj_*^{B}$&$\hat\bbj_*^{C}$&$\hat\bbj_*^{(0)}$&$\hat\bbj_*^{(5)}$&${\hat\bbj_*^{A}}$& $\hat\bbj_*^{B}$&$\hat\bbj_*^{C}$&$\hat\bbj_*^{(0)}$&$\hat\bbj_*^{(5)}$\\ \hline
U-S & 0 & 1 & 0 & 384 & 17 & 0 & 1 & 0 & 0 & 0 \\
\hline
T-S & 0 & 841 & 1 & 615 & 939 & 0 & 999 & 0 & 1000 & 944 \\
\hline
O-S & 1000 & 158 & 999 & 1 & 44 & 1000 & 0 & 1000 & 0 & 56 \\

A-S & 15.86 & 1.09 & 7.40 & 1 & 1.07 & 98.76 & -- & 13.10 & -- & 1.02 \\
\hline
\cline{1-11}
		 &\multicolumn{5}{c|}{$n=1000$} &\multicolumn{5}{c|}{$n=2000$} \\ \cline{2-11}
		& ${\hat\bbj_*^{A}}$& $\hat\bbj_*^{B}$&$\hat\bbj_*^{C}$&$\hat\bbj_*^{(0)}$&$\hat\bbj_*^{(5)}$&${\hat\bbj_*^{A}}$& $\hat\bbj_*^{B}$&$\hat\bbj_*^{C}$&$\hat\bbj_*^{(0)}$&$\hat\bbj_*^{(5)}$\\
\hline
U-S & 0 & 6 & 0 & 0 & 0 & 0 & 374 & 0 & 0 & 0 \\
\hline
T-S & 0 & 994 & 0 & 1000 & 954 & 0 & 626 & 197 & 999 & 954 \\
\hline
O-S & 1000 & 0 & 1000 & 0 & 46 & 1000 & 0 & 803 & 1 & 46 \\

A-S & 208.66 & -- & 7.73 & -- & 1.02 & 438.75 & -- & 2.01 & 1 & 1.02 \\
\hline
		\hline
\multicolumn{11}{|c|}{$\alpha=0.4$, $c=0.4$} \\
\cline{1-11}
		 &\multicolumn{5}{c|}{$n=100$} &\multicolumn{5}{c|}{$n=500$} \\ \cline{2-11}
		& ${\hat\bbj_*^{A}}$& $\hat\bbj_*^{B}$&$\hat\bbj_*^{C}$&$\hat\bbj_*^{(0)}$&$\hat\bbj_*^{(5)}$&${\hat\bbj_*^{A}}$& $\hat\bbj_*^{B}$&$\hat\bbj_*^{C}$&$\hat\bbj_*^{(0)}$&$\hat\bbj_*^{(5)}$\\ \hline
U-S & 0 & 263 & 0 & 969 & 684 & 0 & 994 & 0 & 1 & 0 \\
\hline
T-S & 0 & 540 & 0 & 31 & 302 & 0 & 6 & 0 & 999 & 948 \\
\hline
O-S & 1000 & 197 & 1000 & 0 & 14 & 1000 & 0 & 1000 & 0 & 52 \\

A-S & 30.52 & 1.38 & 28.73 & -- & 1.07 & 194.24 & -- & 192.21 & -- & 1.04 \\
\hline
\cline{1-11}
		 &\multicolumn{5}{c|}{$n=1000$} &\multicolumn{5}{c|}{$n=2000$} \\ \cline{2-11}
		& ${\hat\bbj_*^{A}}$& $\hat\bbj_*^{B}$&$\hat\bbj_*^{C}$&$\hat\bbj_*^{(0)}$&$\hat\bbj_*^{(5)}$&${\hat\bbj_*^{A}}$& $\hat\bbj_*^{B}$&$\hat\bbj_*^{C}$&$\hat\bbj_*^{(0)}$&$\hat\bbj_*^{(5)}$\\
\hline
U-S & 0 & 1000 & 0 & 0 & 0 & 0 & 1000 & 0 & 0 & 0 \\
\hline
T-S & 0 & 0 & 0 & 999 & 953 & 0 & 0 & 0 & 999 & 950 \\
\hline
O-S & 1000 & 0 & 1000 & 1 & 47 & 1000 & 0 & 1000 & 1 & 50 \\

A-S & 394.98 & -- & 394.64 & 1 & 1.02 & 795 & -- & 795 & 1 & 1.02 \\
\hline

	\end{tabular}
	\caption{ Selection times of the KOO methods with AIC, BIC, $C_p$ thresholds and bootstrap methods under Settings  (II)  and (v) based on 1,000 replications. Here U-S, T-S, O-S and A-S stand for number of times a selection method under-specified the true model,  number of times a selection method identified the true model exactly, number of times a selection method over-specified the true model,  and  the average number of spurious variables a selection method identified when it over-specified the model, respectively.
	}\label{KOO5}
\end{table}

 \begin{table}[htbp] \center
	\begin{tabular}{|c|ccccc|ccccc|}
		\hline
\multicolumn{11}{|c|}{$\alpha=0.2$, $c=0.2$} \\
\cline{1-11}
		 &\multicolumn{5}{c|}{$n=100$} &\multicolumn{5}{c|}{$n=500$} \\ \cline{2-11}
		& ${\hat\bbj_*^{A}}$& $\hat\bbj_*^{B}$&$\hat\bbj_*^{C}$&$\hat\bbj_*^{(0)}$&$\hat\bbj_*^{(5)}$&${\hat\bbj_*^{A}}$& $\hat\bbj_*^{B}$&$\hat\bbj_*^{C}$&$\hat\bbj_*^{(0)}$&$\hat\bbj_*^{(5)}$\\

\hline
U-S & 0 & 61 & 0 & 0 & 0 & 0 & 788 & 0 & 0 & 0 \\
\hline
T-S & 483 & 939 & 534 & 999 & 974 & 960 & 212 & 977 & 1000 & 956 \\
\hline
O-S & 517 & 0 & 466 & 1 & 26 & 40 & 0 & 23 & 0 & 44 \\

A-S & 1.49 & -- & 1.44 & 1 & 1.08 & 1 & -- & 1 & -- & 1 \\
\hline

\cline{1-11}
		 &\multicolumn{5}{c|}{$n=1000$} &\multicolumn{5}{c|}{$n=2000$} \\ \cline{2-11}
		& ${\hat\bbj_*^{A}}$& $\hat\bbj_*^{B}$&$\hat\bbj_*^{C}$&$\hat\bbj_*^{(0)}$&$\hat\bbj_*^{(5)}$&${\hat\bbj_*^{A}}$& $\hat\bbj_*^{B}$&$\hat\bbj_*^{C}$&$\hat\bbj_*^{(0)}$&$\hat\bbj_*^{(5)}$\\
\hline
U-S & 0 & 1000 & 0 & 0 & 0 & 0 & 1000 & 0 & 0 & 0 \\
\hline
T-S & 1000 & 0 & 1000 & 1000 & 937 & 1000 & 0 & 1000 & 999 & 951 \\
\hline
O-S & 0 & 0 & 0 & 0 & 63 & 0 & 0 & 0 & 1 & 49 \\

A-S & -- & -- & -- & -- & 1.10 & -- & -- & -- & 1 & 1.02 \\
\hline
\hline
\multicolumn{11}{|c|}{$\alpha=0.2$, $c=0.4$} \\
\cline{1-11}
		 &\multicolumn{5}{c|}{$n=100$} &\multicolumn{5}{c|}{$n=500$} \\ \cline{2-11}
		& ${\hat\bbj_*^{A}}$& $\hat\bbj_*^{B}$&$\hat\bbj_*^{C}$&$\hat\bbj_*^{(0)}$&$\hat\bbj_*^{(5)}$&${\hat\bbj_*^{A}}$& $\hat\bbj_*^{B}$&$\hat\bbj_*^{C}$&$\hat\bbj_*^{(0)}$&$\hat\bbj_*^{(5)}$\\
		\hline
U-S & 0 & 987 & 0 & 161 & 13 & 0 & 1000 & 0 & 0 & 0 \\
\hline
T-S & 60 & 13 & 0 & 838 & 951 & 28 & 0 & 0 & 1000 & 958 \\
\hline
O-S & 940 & 0 & 1000 & 1 & 36 & 972 & 0 & 1000 & 0 & 42 \\

A-S & 3.01 & -- & 6.83 & 1 & 1.06 & 3.94 & -- & 45.52 & -- & 1.05 \\
\hline
\cline{1-11}
		 &\multicolumn{5}{c|}{$n=1000$} &\multicolumn{5}{c|}{$n=2000$} \\ \cline{2-11}
		& ${\hat\bbj_*^{A}}$& $\hat\bbj_*^{B}$&$\hat\bbj_*^{C}$&$\hat\bbj_*^{(0)}$&$\hat\bbj_*^{(5)}$&${\hat\bbj_*^{A}}$& $\hat\bbj_*^{B}$&$\hat\bbj_*^{C}$&$\hat\bbj_*^{(0)}$&$\hat\bbj_*^{(5)}$\\
\hline
U-S & 0 & 1000 & 0 & 0 & 0 & 0 & 1000 & 0 & 0 & 0 \\
\hline
T-S & 237 & 0 & 0 & 999 & 949 & 870 & 0 & 0 & 999 & 954 \\
\hline
O-S & 763 & 0 & 1000 & 1 & 51 & 130 & 0 & 1000 & 1 & 46 \\

A-S & 1.92 & -- & 94.99 & 1 & 1 & 1.05 & -- & 193.56 & 1 & 1 \\
\hline
\hline
\multicolumn{11}{|c|}{$\alpha=0.4$, $c=0.2$} \\
\cline{1-11}
		 &\multicolumn{5}{c|}{$n=100$} &\multicolumn{5}{c|}{$n=500$} \\ \cline{2-11}
		& ${\hat\bbj_*^{A}}$& $\hat\bbj_*^{B}$&$\hat\bbj_*^{C}$&$\hat\bbj_*^{(0)}$&$\hat\bbj_*^{(5)}$&${\hat\bbj_*^{A}}$& $\hat\bbj_*^{B}$&$\hat\bbj_*^{C}$&$\hat\bbj_*^{(0)}$&$\hat\bbj_*^{(5)}$\\
		\hline
U-S & 0 & 0 & 0 & 3 & 0 & 0 & 0 & 0 & 0 & 0 \\
\hline
T-S & 0 & 995 & 17 & 996 & 941 & 0 & 1000 & 82 & 998 & 943 \\
\hline
O-S & 1000 & 5 & 983 & 1 & 59 & 1000 & 0 & 918 & 2 & 57 \\

A-S & 16.54 & 1 & 4.84 & 1 & 1.07 & 102.68 & -- & 2.96 & 1 & 1.05 \\
 \hline
\cline{1-11}
		 &\multicolumn{5}{c|}{$n=1000$} &\multicolumn{5}{c|}{$n=2000$} \\ \cline{2-11}
		& ${\hat\bbj_*^{A}}$& $\hat\bbj_*^{B}$&$\hat\bbj_*^{C}$&$\hat\bbj_*^{(0)}$&$\hat\bbj_*^{(5)}$&${\hat\bbj_*^{A}}$& $\hat\bbj_*^{B}$&$\hat\bbj_*^{C}$&$\hat\bbj_*^{(0)}$&$\hat\bbj_*^{(5)}$\\
\hline
U-S & 0 & 3 & 0 & 0 & 0 & 0 & 345 & 0 & 0 & 0 \\
\hline
T-S & 0 & 997 & 655 & 1000 & 946 & 0 & 655 & 992 & 1000 & 951 \\
\hline
O-S & 1000 & 0 & 345 & 0 & 54 & 1000 & 0 & 8 & 0 & 49 \\

A-S & 217.91 & -- & 1.17 & -- & 1.07 & 465.89 & -- & 1 & -- & 1.06 \\
  \hline
\hline
\multicolumn{11}{|c|}{$\alpha=0.4$, $c=0.4$} \\
\cline{1-11}
		 &\multicolumn{5}{c|}{$n=100$} &\multicolumn{5}{c|}{$n=500$} \\ \cline{2-11}
		& ${\hat\bbj_*^{A}}$& $\hat\bbj_*^{B}$&$\hat\bbj_*^{C}$&$\hat\bbj_*^{(0)}$&$\hat\bbj_*^{(5)}$&${\hat\bbj_*^{A}}$& $\hat\bbj_*^{B}$&$\hat\bbj_*^{C}$&$\hat\bbj_*^{(0)}$&$\hat\bbj_*^{(5)}$\\
		\hline
U-S & 0 & 237 & 0 & 961 & 464 & 0 & 997 & 0 & 0 & 0 \\
\hline
T-S & 0 & 658 & 0 & 39 & 505 & 0 & 3 & 0 & 1000 & 952 \\
\hline
O-S & 1000 & 105 & 1000 & 0 & 31 & 1000 & 0 & 1000 & 0 & 48 \\

A-S & 32.02 & 1.32 & 30.24 & -- & 1.13 & 194.90 & -- & 194.26 & -- & 1.02 \\
\hline
\cline{1-11}
		 &\multicolumn{5}{c|}{$n=1000$} &\multicolumn{5}{c|}{$n=2000$} \\ \cline{2-11}
		& ${\hat\bbj_*^{A}}$& $\hat\bbj_*^{B}$&$\hat\bbj_*^{C}$&$\hat\bbj_*^{(0)}$&$\hat\bbj_*^{(5)}$&${\hat\bbj_*^{A}}$& $\hat\bbj_*^{B}$&$\hat\bbj_*^{C}$&$\hat\bbj_*^{(0)}$&$\hat\bbj_*^{(5)}$\\
\hline
U-S & 0 & 1000 & 0 & 0 & 0 & 0 & 1000 & 0 & 0 & 0 \\
\hline
T-S & 0 & 0 & 0 & 999 & 937 & 0 & 0 & 0 & 999 & 939 \\
\hline
O-S & 1000 & 0 & 1000 & 1 & 63 & 1000 & 0 & 1000 & 1 & 61 \\

A-S & 395 & -- & 394.98 & 1 & 1.02 & 795 & -- & 795 & 1 & 1.05 \\
            \hline

	\end{tabular}
	\caption{ Selection times of the KOO methods with AIC, BIC, $C_p$ thresholds and bootstrap methods under Settings (II) and (vi) based on 1,000 replications. Here U-S , T-S, O-S and A-S stand for number of times a selection method under-specified the true model,  number of times a selection method identified the true model exactly, number of times a selection method over-specified the true model,  and  the average number of spurious variables a selection method identified when it over-specified the model, respectively.
	}\label{KOO6}
\end{table}

\section{Proofs of Theorems \ref{limit1}--\ref{clt2}}

In this appendix, we present the proofs of Theorems \ref{limit1}--\ref{clt2} under general distributions by random matrix theory. Before that, we first give some notation and preliminary results which will be used in the sequel frequently.
For simplicity, we
denote $\bbM=p^{-1}\bbE'\bbQ\bbE$ and $\bbM_l=\frac1p\bbE_l'\bbQ\bbE_l$, where $\bbE_l$ is the $n\times (p-1)$ submatrix of $\bbE$ with the $l$-th column removed. Denote by $\E_l$ the conditional expectation given $\{\bbe_1,\dots,\bbe_l\}$ and by $\E_0=\E$ the unconditional expectation, where $\bbe_i $ is the $n$-vector of the $i$-th column
of $\bbE$. Let $\bbb=\bgS^{-1/2}\bm\Theta_*\bbX_*'\bba_{1}$ and $\bbb_{l}$ be the $p-1$ sub-vector of $\bbb$ with the $l$-th entry $b_{l}$ removed. Then we have
\begin{align*}
 \bba_{1}'\bbY\bgS^{-1/2}(\bbE'\bbQ\bbE)^{-1}\bgS^{-1/2}\bbY'\bba_{1}=p^{-1}(\bbb'+\bba_{1}'\bbE)\bbM^{-1}(\bbE'\bba_{1}+\bbb).
\end{align*}

Modifying the truncation argument of \cite{BaiC18L}, we can assume that the variables $\{e_{ij},i=1\dots n,j=1\dots p\}$ satisfy the following additional condition:
\begin{align}\label{addassum}
	 |e_{ij}|<\eta_n\sqrt{n}, \quad \mbox{for all $i,j$},
	 \end{align}
	 where  $\eta_n\to0$ slowly enough.
By the theorem in the appendix of \cite{BaiS04C}, we know for any positive constant $d<(1-\sqrt{c})^2$ and any given $t>0$, $\lambda_{\min}^{\frac{1}{n}\bbE'\bbE}\asto (1-\sqrt{c})^2$ and
\begin{align*}
 \P(\lambda_{\min}^{\frac{1}{n}\bbE'\bbE}<d)=o(n^{-t}).
\end{align*}
Moreover, by Theorem 1.2 in \cite{BaiS99E}, we conclude that for any positive constant $d<(1-\sqrt{c/(1-\alpha)})^2$ and any given $t>0$, $\lambda_{\min}^{\frac{1}{n}\bbE'\bbQ\bbE}\asto (1-\sqrt{c/(1-\alpha)})^2$ and
\begin{align*}
 \P(\lambda_{\min}^{\frac{1}{n}\bbE'\bbQ\bbE}<d)=o(n^{-t}).
\end{align*}
Denote
 \begin{align*}
 \beta_l&=\frac1p\bbe_l'\bbQ\bbe_l-\frac1{p^2}\bbe_l'\bbQ\bbE_l\bbM_l^{-1}\bbE_l'\bbQ\bbe_l
 \end{align*}
 and
$$\beta^{tr}_1=\tr[\frac1p\bbQ-\frac1{p^2}\bbQ\bbE_l\bbM_l^{-1}\bbE_l'\bbQ]=\frac{n-k-p+1}{p}.$$
 It follows that
 \begin{align}\label{betr}
	\frac1{\beta_l}
	=\frac1{\beta^{tr}_1}-\frac{\xi_l}{\beta_l\beta^{tr}_1},
\end{align}
where $\xi_l=\beta_l-\beta^{tr}_1$.
By Lemma 7.2 in \cite{BaiY05C} (see Lemma \ref{BYinq}), we have that for any $2\leq\ell\leq\log(n),$
\begin{align}\label{Exil}
 \E|\xi_l|^\ell=O(p^{-1}\eta_n^{2\ell-4}),
\end{align}
which indicate that $\xi_l$ tends to 0 in probability with order of $o(n^{-t})$ for any $t>0$.
Analogously, for application later, together with the condition that $p^{-1/2}\bbb$ is bounded in Euclidean norm,
we conclude that for $2\leq\ell\leq\log(n),$
\begin{gather}\label{e1}
\max\{\E|p^{-1}
\bbb_{l}'\bbM_l^{-1}\bbE_l'\bbQ\bbe_l|^{\ell}, \E|p^{-1}
\bba_{l}\bbE_l'\bbM_l^{-1}\bbE_l'\bbQ\bbe_l|^{\ell}\}=O(p^{{\ell/2-1}}\eta_n^{\ell-2})
\end{gather}
and
\begin{gather}\label{Psl}
\max\{\E|p^{-2}\bbb_{l}'\bmPs_l\bbb_{l}|^{\ell},\E|p^{-2}\bba_{1}\bbE_l'\bmPs_l\bbb_{l}|^{\ell},\E|p^{-2}\bba_{1}\bbE_l'\bmPs_l\bbE_l\bba_{1}|^{\ell} \}=O(p^{{\ell-2}}\eta_n^{2\ell-4})
\end{gather}
where
$$\bmPs_l=\bbM_l^{-1}\bbE_l'\bbQ\bbe_l\bbe_l'\bbQ\bbE_l\bbM_l^{-1}-\bbM_l^{-1}\bbE_l'\bbQ\bbE_l\bbM_l^{-1}.$$
As we only need to prove the weak convergence conclusion and $\beta^{tr}_1\to(1-\alpha-c)/c>0$,
thus throughout the proofs, we can safely assume $\|\bbM^{-1}\|$, $\|\bbM_l^{-1}\|$ and $|1/\beta_l|$ are all bounded for large $n$.

\subsection{Proof of Theorem \ref{limit1}}
Theorem \ref{limit1} can be obtained from Proposition 3.1 in \cite{BaiC22A} with letting $z\downarrow0$ directly. That is, for any non-random vectors $\bbr_1$, $\bbr_2$ $\bbr_3$ and $\bbr_4$ with suitable dimensions and bounded in Euclidean norm, under conditions in Theorem \ref{limit1}, we have that for any $t>0$ and $\varepsilon>0$,
\begin{align}
\label{le3.1.1}
\P\left(\left|\bbr'_1\bbM^{-1}\bbr_2-\frac{c_n\bbr'_1\bbr_2}{1-c_n-\alpha_n}\right|\geq \varepsilon\right)=o(n^{-t}),
\end{align}
\begin{align}\label{le3.1.2}
\P\left(\left|\frac{1}{\sqrt{p}}\bbr_1'\bbM^{-1}\bbE'\bbr_3\right|\geq\varepsilon\right)=o(n^{-t}),
\end{align}
and
\begin{align}\label{le3.1.3}
\P\left(\left|\frac{1}{p}\bbr'_3\bbE\bbM^{-1}\bbE'\bbr_4
-\frac{c_n\bbr'_3\bbr_4}{1-c_n-\alpha_n}+\frac{c_n^2\bbr'_3\bbQ\bbr_4}{(1-c_n-\alpha_n)(1-\alpha_n)}\right|\geq \varepsilon\right)
=o(n^{-t}).
\end{align}
Then the proof of Theorem \ref{limit1} is complete.

\subsection{Proof of Theorem \ref{clt1}}

 For simple presentation, in the following we assume $\{1,\dots,q\}\subset [k]\backslash\bbj_*$ and $j_i=i$.
To prove Theorem \ref{clt1}, it is sufficient to show that for any non-null vector $\bbh=(h_1,\dots,h_q)'$, $\sqrt{p}[(\cK_{1},\dots,\cK_{q})\bbh-\frac{c_n}{1-c_n-\alpha_n}{\bf1}_q'\bbh]
$ converges weakly to a normal distribution with mean zero and variance $\frac{c^2}{(1-\alpha_n-c_n)^2}[\frac{2(1-\alpha_n)}{(1-\alpha_n-c_n)}\bbh'(\cA_q'\cA_q)^2\bbh+\tau\bbh'(\cA_q\circ\cA_q)'(\cA_q\circ\cA_q)\bbh]$, where $\cA_q=(\bba_{1},\dots,\bba_{q})$.

 We split the proof of this theorem into two parts. First, we show the asymptotic normality of the sequence of random variables
 \begin{align*}
 	\cM_{1}^{(n)}:=\sqrt{p}[(\cK_1,\dots,\cK_q)\bbh-\E(\cK_1,\dots,\cK_q)\bbh].
 \end{align*}
 Second, we prove the non-random sequence
 \begin{align*}
 	\cM_{2}^{(n)}&=\sqrt{p}[\E(\cK_1,\dots,\cK_q)\bbh-\frac{c_n{\bf1}_q'\bbh}{1-c_n-\alpha_n}]
 \end{align*} tends to zero.
 Note that for notational simplicity the superscript $^{(n)}$ in $\cM_{1}^{(n)}$ and $\cM_{2}^{(n)}$ are suppressed in the sequel.

 We start to
 consider $\cM_1$. Let $\bcH=Diag(h_1,\dots,h_q)$. It follows that
 \begin{align*}
 	\cM_1=&p^{-1/2}\sum_{l=1}^p(\E_l-\E_{l-1})\tr(\bbE\bbM^{-1}\bbE'\cA_q\bcH\cA_q')\\
 	=&p^{-1/2}\sum_{l=1}^p(\E_l-\E_{l-1})\tr[(\bbE \bbM^{-1}\bbE'-\bbE_l \bbM_l^{-1}\bbE_l')\cA_q\bcH\cA_q'].
 \end{align*}
 By the inversion formula of block matrix, we obtain
 \begin{align}
 	& \bbE\bbM^{-1}\bbE'
 	-\bbE_l\bbM_l^{-1}\bbE_l'
 	=\frac{1}{\beta_lp^2} \bbE_l\bbM_l^{-1}\bbE_l'\bbQ\bbe_l\bbe_l'\bbQ\bbE_l\bbM_l^{-1}\bbE_l'\nonumber\\
 	&-\frac{1}{\beta_lp} \bbE_l\bbM_l^{-1}\bbE_l'\bbQ\bbe_l\bbe_l'-\frac{1}{\beta_lp} \bbe_l\bbe_l'\bbQ\bbE_l\bbM_l^{-1}\bbE_l'
 	+\frac{\bbe_l\bbe_l'}{\beta_l}.\label{blk}
 \end{align}
 Then, by the equation \eqref{betr}, 
 we can rewrite $ \cM_{1}$ as
 \begin{align*}
 	\cM_1=\frac{1}{p^{1/2}\beta^{tr}_1}\sum_{l=1}^p\E_l({\bbe_l'\bmGa_l\bbe_l-\tr\bmGa_l})-\frac{1}{p^{1/2}(\beta^{tr}_1
 	)^2}\sum_{l=1}^p\E_l({\xi_l\tr\bmGa_l})+\cM_{10},
 \end{align*}
 where
 \begin{align*}
\bmGa_l=&p^{-2}\bbQ\bbE_l\bbM_l^{-1}\bbE_l'\cA_q\bcH\cA_q'\bbE_l\bbM_l^{-1}\bbE_l'\bbQ
 	-p^{-1}\cA_q\bcH\cA_q'\bbE_l\bbM_l^{-1}\bbE_l'\bbQ\\
 	&-p^{-1}\bbQ\bbE_l\bbM_l^{-1} \bbE_l\cA_q\bcH\cA_q'+\cA_q\bcH\cA_q'
 \end{align*}
 and
 \begin{align*}
 	\cM_{10}=&-\sum_{l=1}^p(\E_l-\E_{l-1})\frac{\xi_l(\bbe_l'\bmGa_l\bbe_l-\tr\bmGa_l)}{p^{1/2}(\beta^{tr}_1)^2}+\sum_{l=1}^p(\E_l-\E_{l-1})\frac{\xi_l^2\bbe_l'\bmGa_l\bbe_l}{p^{1/2}\beta_l(\beta^{tr}_1)^2}.
 \end{align*}
 It follows from \eqref{Exil} that
 \begin{align*}
 \E|\frac{1}{p^{1/2}}\sum_{l=1}^p\E_l({\xi_l}\tr\bmGa_l)|^2= \frac{1}{p}\sum_{l=1}^p\E|\E_l({\xi_l}\tr\bmGa_l)|^2=O(p^{-1}).
\end{align*}
 By \eqref{e1}, \eqref{Psl} and the BurkH\"{o}lder's inequality (see Lemma \ref{burkholder}) we have that $\cM_{10}=o_p(1)$.
 Applying Lemma 2.7 in \cite{BaiS98N} (see Lemma \ref{BSinq}), we have that
 \begin{align*}
 	\E|\E_l(\bbe_l'\bmGa_l\bbe_l-\tr\bmGa_l)|^4\leq \E|\bbe_l'\bmGa_l\bbe_l-\tr\bmGa_l|^4=O(p\eta_n^4)
 \end{align*}
 which verifies the condition (ii) in Lemma \ref{martingaleCLT}. Thus, what we need is to obtain the limit of
 \begin{align*}
 	\frac{1}{p(\beta^{tr}_1)^2}\sum_{l=1}^p\E_{l-1}\{\E_l[{\bbe_l'\bmGa_l\bbe_l-\tr\bmGa_l}]\}^2.
 \end{align*}
 By Lemma \ref{quadratic_equ}, we have that
 \begin{align*}
 	\E_{l-1}\{\E_l[{\bbe_l'\bmGa_l\bbe_l-\tr\bmGa_l}]\}^2
 	=2\E_{l-1}\tr(\E_l\bmGa_l\E_l\bmGa_l) +\tau\E_{l-1}\tr(\E_l\bmGa_l\circ\E_l\bmGa_l),
 \end{align*}
 where $\circ$ stands for the Hadamard product.
 Notice that
 \begin{align}
\tr(\E_l\bmGa_l\E_l\bmGa_l)=&p^{-4}\tr(\E_l\bbQ\bbE_l\bbM_l^{-1}\bbE_l'\cA_q\bcH\cA_q'\bbE_l\bbM_l^{-1}\bbE_l'\bbQ)^2\label{trElGa21}\\
 	&+p^{-2}2\tr[\E_l(\cA_q\bcH\cA_q'\bbE_l\bbM_l^{-1}\bbE_l'\bbQ)\E_l(\bbQ\bbE_l\bbM_l^{-1} \bbE_l\cA_q\bcH\cA_q')]\label{trElGa22}\\
 	&+\tr(\cA_q\bcH\cA_q')^2.\nonumber
 \end{align}
 		Let $\widetilde\bbE_l$ be $\bbE_l$ by replacing $\{\bbe_{l+1},\dots,\bbe_{p}\}$ with $\{\tilde\bbe_{l+1},\dots,\tilde\bbe_{	p}\}$, where $\{\tilde\bbe_i\}$ are i.i.d. copies of $\bbe_1$. We define $\widetilde\bbM_l=\frac1p\widetilde\bbE_l'\bbQ\widetilde\bbE_l$, correspondingly. As $\bcH$ is a diagonal matrix, thus we have that
 		\begin{align*}
 			&\E_{l-1} \tr(\E_l\bbQ\bbE_l\bbM_l^{-1}\bbE_l'\cA_q\bcH\cA_q'\bbE_l\bbM_l^{-1}\bbE_l'\bbQ)^2\nonumber\\
 			=&\E_l\tr[\bcH\cA_q'\bbE_l\bbM_l^{-1}\bbE_l'\widetilde{\bm\Xi}_l\bbE_l\bbM_l^{-1}\bbE_l'\cA_q]\\
 			=&\E_l\sum_{i=1}^qh_i\bba_i'\bbE_l\bbM_l^{-1}\bbE_l'\widetilde{\bm\Xi}_l\bbE_l\bbM_l^{-1}\bbE_l'\bba_i,
 		\end{align*}
 		where $\widetilde{\bm\Xi}_l=\bbQ\widetilde\bbE_l\widetilde\bbM_l^{-1}\widetilde\bbE_l'\cA_q\bcH\cA_q'\widetilde\bbE_l\widetilde\bbM_l^{-1}\widetilde\bbE_l'\bbQ$.
 		By applying the inversion formula of block matrix to $\bbM_l$, similar to \eqref{blk}, we have that
 		\begin{align}\label{beamel}
 			&\bbE_l\bbM_l^{-1}\bbE_l'-\bbE_{lp}\bbM_{lp}^{-1}\bbE_{lp}'
 			=\frac{1}{\beta_{lp}p^2} \bbE_{lp}\bbM_{lp}^{-1}\bbE_{lp}'\bbQ\bbe_p\bbe_p'\bbQ\bbE_{lp}\bbM_{lp}^{-1}\bbE_{lp}'\no\\
 			&
 			-\frac{1}{\beta_{lp}p}\bbE_{lp}\bbM_{lp}^{-1}\bbE_{lp}'\bbQ\bbe_p\bbe_p'-\frac{1}{\beta_{lp}p}\bbe_p\bbe_p'\bbQ\bbE_{lp}\bbM_{lp}^{-1}\bbE_{lp}'+\frac{\bbe_p\bbe_p'}{\beta_{lp}},
 		\end{align}
 		where $\bbE_{li}$ is the $n\times (i-2)$ submatrix of $\bbE$ with the columns $\{\bbe_l, \bbe_i,\dots,\bbe_p\}$ removed, $\bbM_{li}=\frac1p\bbE_{li}'\bbQ\bbE_{li}$ and
 		\begin{align}
 			\beta_{li}&=\frac1p\bbe_i'\bbQ\bbe_i-\frac1{p^2}\bbe_i'\bbQ\bbE_{li}\bbM_{li}^{-1}\bbE_{li}'\bbQ\bbe_i.
 		\end{align}
 		Denote
 		$$\beta^{tr}_i=\tr[\frac1p\bbQ-\frac1{p^2}\bbQ\bbE_{li}\bbM_{li}^{-1}\bbE_{li}'\bbQ]=\frac{n-k-p+i}{p}$$
 		and
 		$$\xi_{li}=\beta_{li}-\beta^{tr}_i.$$
 		We can easily check that the orders of
 		\eqref{Exil}--\eqref{Psl} hold for replacing the subscripts $l$ by $li$. Thus, analogous to the above discussion, we have that
 		\begin{align*}
&\E_l\bba_i'\bbE_l\bbM_l^{-1}\bbE_l'\widetilde{\bm\Xi}_l\bbE_l\bbM_l^{-1}\bbE_l'\bba_i\\
 			=&\E_l[\bba_i'\bbE_{lp}\bbM_{lp}^{-1}\bbE_{lp}'\widetilde{\bm\Xi}_l\bbE_{lp}\bbM_{lp}^{-1}\bbE_{lp}'\bba_i]+o_p(p^3)\\
 			=&\E_l[\bba_i'\bbE_{lp}\bbM_{lp}^{-1}\bbE_{lp}'\widetilde{\bm\Xi}_{lp}\bbE_{lp}\bbM_{lp}^{-1}\bbE_{lp}'\bba_i]+o_p(p^3),
 		\end{align*}
 		where $\widetilde{\bm\Xi}_{lp}$ is defined by removing $b_p$ and $\tilde \bbe_p$ from $\widetilde{\bm\Xi}_{l}$. We then repeat the procedure that remove $b_i$, $\bbe_i$ and $\tilde \bbe_i$, $i=l+1,\dots,p-1$ from ${\bm\Xi}_{lp}$ and $\widetilde{\bm\Xi}_{lp}$, respectively. Then applying Proposition 3.1 in \citep{BaiC22A}, we finally obtain that
 		\begin{align}
 			\eqref{trElGa21}
 			=&\sum_{i,j}^qh_ih_j
 			p^{-2}[\bba_i'\bbE_{l(l-1)}\bbM_{l(l-1)}^{-1}\bbE_{l(l-1)}'\bba_j)]^2+o_p(1 )\nonumber\\
 			=&\frac{(l-1)^2}{(n-k-l+1)^2}\bbh'(\cA_q'\cA_q)^2\bbh+o_p(1).\label{p4Mlp}
 		\end{align}

 		Analogously, we have that
 		\begin{align*}
 			\eqref{trElGa22}=&p^{-2}\sum_{i=1}^qh_ih_j\E_l(\bba_i'\bbE_l\bbM_l^{-1}\bbE_l'\bbQ)\E_l(\bbQ\bbE_l\bbM_l^{-1} \bbE_l\bba_i)\\
 			=&\frac{(l-1)\bbh'(\cA_q'\cA_q)^2\bbh}{n-k-l+1}+o_p(1),
 		\end{align*}
 		which together with \eqref{p4Mlp} and the fact that $\tr(\cA_q\bcH\cA_q')^2=\bbh'(\cA_q'\cA_q)^2\bbh$ implies
 		\begin{align*}
 			\frac{1}{p}\sum_{l=1}^p \E_{l-1}\tr(\E_l\bmGa_l\E_l\bmGa_l)=&\frac{\bbh'(\cA_q'\cA_q)^2\bbh}{p}\sum_{l=1}^p\left(\frac{l-1}{n-k-l+1}+1\right)^2+o_p(1)\\
 			=&\bbh'(\cA_q'\cA_q)^2\bbh\frac{1-\alpha_n}{1-\alpha_n-c_n}+o_p(1).
 		\end{align*}

 		We now turn to prove the term $\E_{l-1}\tr(\E_l\bmGa_l\circ\E_l\bmGa_l)=o_p(1)$.
 		Let $\bbu_j$ be an $n$-dimensional column vector with the j-th element being 1 and 0 otherwise. Then we have that
 		\begin{align*}
 			&\E(\E_{l-1}\tr(\E_l\bmGa_l\circ\E_l\bmGa_l))
 			=\sum_{j=1}^n\E(\bbu_j'\E_l\bmGa_l\bbu_j)^2\\
 			=&\sum_{j=1}^n(\E\bbu_j'\bmGa_l\bbu_j)^2+\sum_{j=1}^n\E(\E_l\bbu_j'\bmGa_l\bbu_j-\E\bbu_j'\bmGa_l\bbu_j)^2.
 		\end{align*}
 		By BurkH\"{o}lder's inequality, we have that
 		\begin{align*}
 &\E(\E_l\bbu_j'\bmGa_l\bbu_j-\E\bbu_j'\bmGa_l\bbu_j)^2
 \leq \sum_{s\neq l}^p\E(\bbu_j'\bmGa_l\bbu_j-\bbu_j'\bmGa_{l\cdot s}\bbu_j)^2,
\end{align*}
where $\bmGa_{l\cdot s}$ is the submatrix of $\bmGa_{l}$ with $\bbe_s$ removed.
 Applying the inversion formula of block matrix \eqref{beamel} again, we have that
 \begin{align}\label{beamels}
\bbQ\bbE_l\bbM_l^{-1}\bbE_l'=&\bbQ\bbE_{l\cdot s}\bbM_{l\cdot s}^{-1}\bbE_{l\cdot s}'+\frac{\bbQ\bbe_s\bbe_s'}{\beta_{l\cdot s}}\\
+&\frac{1}{\beta_{l\cdot s}p^2} \bbQ\bbE_{l\cdot s}\bbM_{l\cdot s}^{-1}\bbE_{l\cdot s}'\bbQ\bbe_s\bbe_s'\bbQ\bbE_{l\cdot s}\bbM_{l\cdot s}^{-1}\bbE_{l\cdot s}'\no\\
&
 -\frac{1}{\beta_{l\cdot s}p} \bbQ\bbE_{l\cdot s}\bbM_{l\cdot s}^{-1}\bbE_{l\cdot s}'\bbQ\bbe_s\bbe_s'-\frac{1}{\beta_{l\cdot s}p}\bbQ\bbe_s\bbe_s'\bbQ\bbE_{l\cdot s}\bbM_{l\cdot s}^{-1}\bbE_{l\cdot s}'\no\\
:=&\bcU_{ls0}+\bcU_{ls1}+\bcU_{ls2}-\bcU_{l s3}-\bcU_{ls4}\no
\end{align}
and
 		\begin{align*}
 			&\E(\bbu_j'\bmGa_l\bbu_j-\bbu_j'\bmGa_{ls}\bbu_j)^2\\
 			=&\E\{p^{-4}\sum_{i=1}^qh_i[\bbu_j'(\bcU_{ls1}+\bcU_{ls2}-\bcU_{ls3}-\bcU_{ls4})\bba_i]^2\\
 			&+2p^{-2}\sum_{i=1}^qh_i\bbu_j'\bcU_{ls0}\bba_i\bbu_j'(\bcU_{ls1}+\bcU_{ls2}-\bcU_{ls3}-\bcU_{ls4})\bba_i\\
 			&-2p^{-1}\sum_{i=1}^qh_i\bbu_j'\bba_i\bbu_j'(\bcU_{ls1}+\bcU_{ls2}-\bcU_{ls3}-\bcU_{ls4})\bba_i\}^2.
 		\end{align*}
 		
 		We first consider $\E (\bbu_j'\bcU_{ls1}\bba_i\bba_i'\bcU_{ls1}'\bbu_j)^2$. Notice
 		that
 		\begin{align*}
& \E (\bbe_s'\bbQ\bbu_j\bbu_j'{\bbQ\bbe_s\bbe_s'}\bba_i\bba_i'\bbe_s)^2\\
 =&\E [(\bbe_s'\bbQ\bbu_j\bbu_j'\bbQ\bbe_s-\bbu_j'\bbQ\bbu_j+\bbu_j'\bbQ\bbu_j)(\bbe_s'\bba_i\bba_i'\bbe_s-1+1)]^2\\
 =&\E [(\bbe_s'\bbQ\bbu_j\bbu_j'\bbQ\bbe_s-\bbu_j'\bbQ\bbu_j)(\bbe_s'\bba_i\bba_i'\bbe_s-1)+\bbu_j'\bbQ\bbu_j(\bbe_s'\bba_i\bba_i'\bbe_s-1)\\&+(\bbe_s'\bbQ\bbu_j\bbu_j'\bbQ\bbe_s-\bbu_j'\bbQ\bbu_j)+\bbu_j'\bbQ\bbu_j]^2.
\end{align*}
From Lemma \ref{BYinq} we have that for $\ell\geq2$,
\begin{align*}
	\E (\bbe_s'\bbQ\bbu_j\bbu_j'\bbQ\bbe_s-\bbu_j'\bbQ\bbu_j)^\ell =O(n^{\ell-1}\eta_n^{2\ell-4})
\end{align*}
and
\begin{align*}
	\E (\bbe_s'\bba_i\bba_i'\bbe_s-1)^\ell =O(n^{\ell-1}\eta_n^{2\ell-4}).
\end{align*}
 		 	Then, together with the fact that $	\frac1{\beta_{l\cdot s}}$ and $\bbu_j'\bbQ\bbu_j$ are both bounded, and the
 		 		 $c_r$-inequality, we obtain
 		 \begin{align*}
 		 	\E (\bbu_j'\bcU_{ls1}\bba_i\bba_i'\bcU_{ls1}'\bbu_j)^2=O(n^{3}\eta_n^{-4}).	
 		 \end{align*}	
 	
 	 Next, we consider the term $\E (\bbu_j'\bcU_{ls0}\bba_i\bba_i'\bcU_{ls1}'\bbu_j)^2$. It follows $\bba_i'\bbQ=\bf0$ that
 	 \begin{align*}
 	 	&\E(\bbu_j'\bbQ\bbE_{l\cdot s}\bbM_{l\cdot s}^{-1}\bbE_{l\cdot s}'\bba_i\bba_i'\bbe_s\bbe_s'\bbQ\bbu_j)^2\\
 	 	=&\E(\bbe_s'\bbQ\bbu_j\bbu_j'\bbQ\bbE_{l\cdot s}\bbM_{l\cdot s}^{-1}\bbE_{l\cdot s}'\bba_i\bba_i'\bbe_s-\bba_i'\bbQ\bbu_j\bbu_j'\bbQ\bbE_{l\cdot s}\bbM_{l\cdot s}^{-1}\bbE_{l\cdot s}'\bba_i)^2=O(1).
 	 \end{align*}
 	 As other terms are analogous, thus by combining the above argument, we conclude that $\sum_{j=1}^n\E(\E_l\bbu_j'\bmGa_l\bbu_j-\E\bbu_j'\bmGa_l\bbu_j)^2=o(1)$.

 	 For $\sum_{j=1}^n(\E\bbu_j'\bmGa_l\bbu_j)^2$, it follows from the assumption that
 	 $\{e_{ij}\}$ are i.i.d., \begin{align*}
 \sum_{j=1}^n(\E\bbu_j'\bmGa_l\bbu_j)^2=&\sum_{j=1}^n(n^{-1}p^{-1}\E \tr \bcH\cA_q'\bbE_l\bbM_l^{-1}\bbE_l'\cA_q+\bbu_j'\cA_q\bcH\cA_q'\bbu_j)^2\\
 =&\sum_{j=1}^n(\bbu_j'\cA_q\bcH\cA_q'\bbu_j)^2+O(n^{-1})\\
 =&\bbh'(\cA_q\circ\cA_q)'(\cA_q\circ\cA_q)\bbh+O(n^{-1}).
\end{align*}
Here we use a result similar to \eqref{le3.1.3}, that is $$\E
\frac{1}{p}\bbr'_3\bbE\bbM^{-1}\bbE'\bbr_4
-\frac{c_n\bbr'_3\bbr_4}{1-c_n-\alpha_n}+\frac{c_n^2\bbr'_3\bbQ\bbr_4}{(1-c_n-\alpha_n)(1-\alpha_n)}\to0,$$ and the proof can be found in the proof of Proposition 3.1 in \cite{BaiC22A}.
 Then we conclude that
 		\begin{align*}
 			\sum_{l=1}^p\E_{l-1}\tr(\E_l\bmGa_l\circ\E_l\bmGa_l)=\bbh'(\cA_q\circ\cA_q)'(\cA_q\circ\cA_q)\bbh+o_p(1).
 		\end{align*}
 		
 		Next, we will prove that the non-random sequence
 		\begin{align*}
 			\cM_{2}= \cM_{2}^{(n)}=o(1).
 		\end{align*}
 		Write $\bbM^{-1}=(M^{ij})$. Without loss of generality, we only need to prove $p^{-1}\E\bba_{1}'\bbE\bbM^{-1}\bbE'\bba_{1}-\frac{c_n}{1-c_n-\alpha_n}=o(p^{-1/2})$. Because the entries of $\bbE$ are i.i.d., we have
 		\begin{align}
 			&p^{-1}\E\bba_{1}'\bbE\bbM^{-1}\bbE'\bba_{1}=p^{-1}\sum_{i,j=1}^p\E\bba_{1}'\bbe_iM^{ij}\bbe_j'\bba_{1}\nonumber\\ 		=&\E\bbe_1'\bba_{1}\bba_{1}'\bbe_1M^{11}+(p-1)\E\bba_{1}'\bbe_1M^{12}\bbe_2'\bba_{1}.\label{ea2}
 		\end{align}
 		From the inverse matrix formula, we know that
 		\begin{align*}
 			M^{11}=\frac1{\beta_1}
 			=\frac1{\beta^{tr}_1}-\frac{\xi_1}{(\beta^{tr}_1
 				)^2}+\frac{\xi_1^2}{\beta_1(\beta^{tr}_1)^2}.
 		\end{align*}
 		and
 		\begin{align}
 			M^{12}
 			=\frac{\bbe_1'\bbQ\bbE_1\bbM_1^{-1}\bbu_1}{p\beta^{tr}_1}-\frac{\xi_1\bbe_1'\bbQ\bbE_1\bbM_1^{-1}\bbu_1}{p(\beta^{tr}_1
 				)^2}+\frac{\xi_1^2\bbe_1'\bbQ\bbE_1\bbM_1^{-1}\bbu_1}{p\beta_1(\beta^{tr}_1)^2}.\label{M12}
 		\end{align}
 		Then it follows from \eqref{betr}, \eqref{Exil} and the H\"{o}lder's inequality that
 		\begin{align*}
 			\E\bbe_1'\bba_{1}\bba_{1}'\bbe_1M^{11}-\frac{c_n}{1-c_n-\alpha_n}=\E\frac{\bbe_1'\bba_{1}\bba_{1}'\bbe_1\xi_1^2}{\beta_1(\beta^{tr}_1)^2}=o(p^{-1/2}).
 		\end{align*}
 		Moreover, substituting
 	\eqref{M12} into the second term of \eqref{ea2}, we have three terms. The first one is $$\E\frac{\bbe_2'\bba_{1}\bba_{1}'\bbe_1\bbe_1'\bbQ\bbE_1\bbM_1^{-1}\bbu_1}{p\beta^{tr}_1}=\E\frac{\bbe_2'\bba_{1}\bba_{1}'\bbQ\bbE_1\bbM_1^{-1}\bbu_1}{p\beta^{tr}_1}=0,$$
 	because of 	$\bba_{1}'\bbQ=\bf0$. 	Applying the inversion formula to $\bbM_1^{-1}$ again, we obtain that
 		\begin{align*}
 			&\E\frac{\xi_1\bbe_1'\bbQ\bbE_1\bbM_1^{-1}\bbu_1\bba_{1}'\bbe_1\bbe_2'\bba_{1}}{p}\\
 =&\E\frac{\xi_1\bbe_1'\bbQ\bbe_2\bba_{1}'\bbe_1\bbe_2'\bba_{1}}{\beta_{1\cdot2}p}
 			 -\E\frac{\xi_1\bbe_1'\bbQ\bbE_{1\cdot2}\bbM_{1\cdot2}^{-1}\bbE_{1\cdot2}'\bbQ\bbe_2\bba_{1}'\bbe_1\bbe_2'\bba_{1}}{\beta_{1\cdot2}p^2} \\			 =&\E\frac{\xi_1\bbe_1'\bbQ\bbe_2\bba_{1}'\bbe_1\bbe_2'\bba_{1}}{\beta_{1\cdot2}p}.
 		\end{align*}
 		Rewrite $1/\beta_{1\cdot2}$ as
 		\begin{align*}
 			\frac1{\beta_{1\cdot2}}=\frac1{\beta^{tr}_{2}}-\frac{\xi_{1\cdot2}}{\beta_{1\cdot2}\beta^{tr}_{2}}
 		\end{align*}
 		and by the the fact that $\bbQ\bba_1=\bf0$, we have that
 		\begin{align*}
 	\E\frac{\xi_1\bbe_1'\bbQ\bbe_2\bba_{1}'\bbe_1\bbe_2'\bba_{1}}{\beta_{1\cdot2}}=-\E\frac{\xi_{1\cdot2}\xi_1\bbe_1'\bbQ\bbe_2\bba_{1}'\bbe_1\bbe_2'\bba_{1}}{\beta^{tr}_{2}\beta_{1\cdot2}}=o(p^{-1/2}).
 		\end{align*}
 		Therefore, by combining the above results, we conclude that
 		\begin{align*}
 			\cM_{2}=o(1),
 		\end{align*}
 		and we complete the proof of the theorem.

 \subsection{Proof of Theorem \ref{hattau}}
Note that $$\tr[(\bbY'\bbQ\bbY-(n-k)\bbI)\circ(\bbY'\bbQ\bbY-(n-k)\bbI)]=\sum_{i=1}^p(\bbe_i'\bbQ\bbe_i-(n-k))^2$$ and $$\E(\bbe_i'\bbQ\bbe_i-(n-k))^2=2(n-k)+\tau\tr(\bbQ\circ\bbQ).$$ Thus by the definition of $\hat\tau$ and $\{\bbe_i\}$ are i.i.d., we have $\E\hat\tau=\tau$. Next we will show that $$\E(\hat\tau-\tau)^2\to0.$$

 It follows from \eqref{addassum} and Lemma \ref{BSinq} that \begin{align*}
 &\E(\hat\tau-\tau)^2=\E(\hat\tau-\E\hat\tau)^2\\
 =&p^{-1}\E[(\bbe_1'\bbQ\bbe_1-(n-k))^2-\E(\bbe_1'\bbQ\bbe_1-(n-k))]^2/\tr^2(\bbQ\circ\bbQ)\\
 =&\{p^{-1}\E(\bbe_1'\bbQ\bbe_1-(n-k))^4-p^{-1}[\E(\bbe_1'\bbQ\bbe_1-(n-k))]^2\}/\tr^2(\bbQ\circ\bbQ)\\
 \leq&Kp^{-1}[((n-k)^2+n^2(n-k)\eta_n^4+(n-k)^2+\tau^2\tr^2(\bbQ\circ\bbQ)]/\tr^2(\bbQ\circ\bbQ),
\end{align*}
where $K$ is a positive constant. By $c_r$-inequality, we have that $\tr(\bbQ\circ\bbQ)\geq n^{-1}(n-k)^2$, which together with condition (C1) implies $\E(\hat\tau-\tau)^2\to0$. Then we complete the proof of this theorem.

 \subsection{Proof of Theorem \ref{clt2}}
 For simple presentation, in the following we assume $\{1\}\subset \bbj_*$ and let $j=1$. Then by the notation $\bbb=\bgS^{-1/2}\bm\Theta_*\bbX_*'\bba_{1}$, $\cK_1=p^{-1}(\bbb'+\bba_{1}'\bbE)\bbM^{-1}(\bbE'\bba_{1}+\bbb)$. Note that
 the proof procedure of Theorem \ref{clt2} is the same as that of Theorem \ref{clt1}. And the difference is that Theorem \ref{clt2} requires the consideration of linear combinations of three different forms of random variables, namely $\bba_{1}'\bbE\bbM^{-1}\bbE'\bba_{1}$, $\bba_{1}'\bbE\bbM^{-1}\bbb$ and $\bbb'\bbM^{-1}\bbb$. As the asymptotic normality of $\bba_{1}'\bbE\bbM^{-1}\bbE'\bba_{1}$ is proved in last subsection, in the sequel we only focus on the other two terms and their correlations.

Analogously, we split the proof of this theorem into two parts. It is worthy noting that next we may use the same notation as in the proof of Theorem \ref{clt1}, but they represent a little different content.
 First, we show the asymptotic normality of the sequence of random variables
\begin{align*}
\cM_{3}:=\sqrt{p}[p^{-1}(\bbb'+\bba_{1}'\bbE)\bbM^{-1}(\bbE'\bba_{1}+\bbb)-\E p^{-1}(\bbb'+\bba_{1}'\bbE)\bbM^{-1}(\bbE'\bba_{1}+\bbb)].
 \end{align*}
 Second, we prove the non-random sequence
 \begin{align*}
 \cM_{4}&=\sqrt{p}[\E p^{-1}(\bbb'+\bba_{1}'\bbE)\bbM^{-1}(\bbE'\bba_{1}+\bbb)-\frac{c_n(1+\delta_1)}{1-c_n-\alpha_n}]
\end{align*} tends to zero.
It follows that
\begin{align*}
 \cM_1=&p^{-1/2}\sum_{l=1}^p(\E_l-\E_{l-1})(\bbb'+\bba_{1}'\bbE)\bbM^{-1}(\bbE'\bba_{1}+\bbb)\\
 =&p^{-1/2}\sum_{l=1}^p(\E_l-\E_{l-1})[(\bbb'+\bba_{1}'\bbE)\bbM^{-1}(\bbE'\bba_{1}+\bbb)-(\bbb_{l}'+\bba_{1}'\bbE_l)\bbM_l^{-1}(\bbE_l'\bba_{1}+\bbb_{l})].
\end{align*}
By the inversion formula of block matrix, we obtain
\begin{align*}
 & (\bbb'+\bba_{1}'\bbE)\bbM^{-1}(\bbE'\bba_{1}+\bbb)
 -(\bbb_l'+\bba_{1}'\bbE_l)\bbM_l^{-1}(\bbE_l'\bba_{1}+\bbb_{l})\nonumber\\
 =&\frac{1}{\beta_lp^2} (\bbb_l'+\bba_{1}'\bbE_l)\bbM_l^{-1}\bbE_l'\bbQ\bbe_l\bbe_l'\bbQ\bbE_l\bbM_l^{-1}(\bbE_l'\bba_{1}+\bbb_{l})\nonumber\\
 &-\frac{2}{\beta_lp} (\bbb_l'+\bba_{1}'\bbE_l)\bbM_l^{-1}\bbE_l'\bbQ\bbe_l(\bbe_l'\bba_1+b_{l})
 +\frac{(\bbe_l'\bba_1+b_{l})^2}{\beta_l}.
\end{align*}
Then, by the equation \eqref{betr}, 
we can rewrite $ \cM_{3}$ as
\begin{align*}
 \cM_3=\frac{1}{p^{1/2}\beta^{tr}_1}\sum_{l=1}^p\E_l[{\bbe_l'\bmGa_l\bbe_l-\tr\bmGa_l}+2\bbe_l'\bm\gamma_l]+\cM_{30},
\end{align*}
where
\begin{align*}
 \bmGa_l=&p^{-2}\bbQ\bbE_l\bbM_l^{-1}(\bbb_l+\bbE_l'\bba_1)(\bbb_l+\bbE_l'\bba_1)'\bbM_l^{-1}\bbE_l'\bbQ\\
 &-p^{-1}\bba_1(\bbb_l+\bbE_l'\bba_1)'\bbM_l^{-1}\bbE_l'\bbQ
 -p^{-1}\bbQ\bbE_l\bbM_l^{-1}(\bbb_l+\bbE_l'\bba_1)\bba_1'+\bba_1\bba_1',
 \end{align*}
\begin{align*}
 \bmga_l=-p^{-1} b_{l}\bbQ\bbE_l\bbM_l^{-1}(\bbb_l+\bbE_l'\bba_{1})+\bba_1b_{l}
\end{align*}
and
\begin{align*}
 \cM_{30}=&-\sum_{l=1}^p(\E_l-\E_{l-1})\frac{\xi_l(\bbb_l'+\bba_{1}'\bbE_l)\bbM_l^{-1}\bbE_l'\bbQ\bbe_l\bbe_l'\bbQ\bbE_l\bbM_l^{-1}(\bbE_l'\bba_{1}+\bbb_{l})}{p^{5/2}\beta_l\beta^{tr}_1}\\
 &+2\sum_{l=1}^p(\E_l-\E_{l-1})\frac{\xi_l(\bbb_l'+\bba_{1}'\bbE_l)\bbM_l^{-1}\bbE_l'\bbQ\bbe_l(\bbe_l'\bba_1+b_{l})}{p^{3/2}\beta_l\beta^{tr}_1}\\
 &-\sum_{l=1}^p(\E_l-\E_{l-1})\frac{\xi_l(\bbe_l'\bba_1+b_{l})^2}{p^{1/2}\beta_l\beta^{tr}_1}\\
 :=&-\cM_{301}+2\cM_{302}-\cM_{303}.
\end{align*}
Next we will prove $\cM_{10}=o_p(1)$. Substitute \eqref{betr} into $\cM_{101}, ~\cM_{102}$ and $\cM_{103}$ respectively, we then have that
\begin{align*}
 \cM_{301} =&\sum_{l=1}^p(\E_l-\E_{l-1})\frac{\xi_l(\bbb_l'+\bba_{1}'\bbE_l)\bmPs_l(\bbE_l'\bba_{1}+\bbb_{l})}{p^{5/2}(\beta^{tr}_1)^2}+\sum_{l=1}^p(\E_l-\E_{l-1})\frac{\xi_l(\bbb_l'+\bba_{1}'\bbE_l)\bbM_l^{-1}(\bbE_l'\bba_{1}+\bbb_{l})}{p^{3/2}(\beta^{tr}_1)^2}\\
 &-\sum_{l=1}^p(\E_l-\E_{l-1})\frac{\xi_l^2(\bbb_l'+\bba_{1}'\bbE_l)\bbM_l^{-1}\bbE_l'\bbQ\bbe_l\bbe_l'\bbQ\bbE_l\bbM_l^{-1}(\bbE_l'\bba_{1}+\bbb_{l})}{p^{5/2}\beta_l(\beta^{tr}_1)^2},\\
 \cM_{302} =&\sum_{l=1}^p(\E_l-\E_{l-1})\frac{\xi_l(\bbb_l'+\bba_{1}'\bbE_l)\bbM_l^{-1}\bbE_l'\bbQ\bbe_l(\bbe_l'\bba_1+b_{l})}{p^{3/2}(\beta^{tr}_1)^2}\\
 &-\sum_{l=1}^p(\E_l-\E_{l-1})\frac{\xi_l^2(\bbb_l'+\bba_{1}'\bbE_l)\bbM_l^{-1}\bbE_l'\bbQ\bbe_l(\bbe_l'\bba_1+b_{l})}{p^{3/2}\beta_l(\beta^{tr}_1)^2},\\
 \cM_{303} =&\sum_{l=1}^p(\E_l-\E_{l-1})\frac{\xi_l(\bbe_l'\bba_1\bba_1'\bbe_l-1+2b_{l}\bba_1'\bbe_l+b_{l}^2+1)}{p^{1/2}(\beta^{tr}_1)^2} -\sum_{l=1}^p(\E_l-\E_{l-1})\frac{\xi_l^2(\bbe_l'\bba_1+b_{l})^2}{p^{1/2}\beta_l(\beta^{tr}_1)^2}.
\end{align*}
These together with \eqref{e1}, \eqref{Psl} and the BurkH\"{o}lder's inequality (see Lemma \ref{burkholder}) implies that $\cM_{10}=o_p(1)$. Note that here we used the fact $\bbQ\bba_1={\bf0}$.

Applying Lemma \ref{BSinq}, we have that
 \begin{align*}
 \E|\E_l(\bbe_l'\bmGa_l\bbe_l-\tr\bmGa_l)|^4\leq \E|\bbe_l'\bmGa_l\bbe_l-\tr\bmGa_l|^4=O(p\eta_n^4)
\end{align*}
and
\begin{align*}
 \E|\E_l(\bbe_l'\bm\gamma_l)|^4\leq \E|\bbe_l'\bm\gamma_l\bm\gamma_l^*\bbe_l|^2=O(p^{-2}b_{l}^4),
\end{align*}
which verify the condition (ii) in Lemma \ref{martingaleCLT}. Thus, what we need is to obtain the limit of
\begin{align*}
 \frac{1}{p(\beta^{tr}_1)^2}\sum_{l=1}^p\E_{l-1}\{\E_l[{\bbe_l'\bmGa_l\bbe_l-\tr\bmGa_l}+2\bbe_l'\bm\gamma_l]\}^2.
\end{align*}
By Lemma \ref{quadratic_equ}, we have that
\begin{align*}
 &\E_{l-1}\{\E_l[{\bbe_l'\bmGa_l\bbe_l-\tr\bmGa_l}+2\bbe_l'\bm\gamma_l]\}^2\\
 =&2\E_{l-1}\tr(\E_l\bmGa_l\E_l\bmGa_l) +4\E_{l-1}(\E_l\bmga_l\E_l\bmga_l')\\
 &+\tau\E_{l-1}\tr(\E_l\bmGa_l\circ\E_l\bmGa_l)+4\E e_{11}^3\E_{l-1}\tr(\E_l\bmGa_l\circ\E_l\bmga_l{\bf1}').
\end{align*}
Notice that
\begin{align}\label{trElGa21}
\tr(\E_l\bmGa_l\E_l\bmGa_l)=&p^{-4}\tr(\E_l\bbQ\bbE_l\bbM_l^{-1}(\bbb_l+\bbE_l'\bba_1)(\bbb_l+\bbE_l'\bba_1)'\bbM_l^{-1}\bbE_l'\bbQ)^2\\
&+p^{-2}2\E_l((\bbb_l+\bbE_l'\bba_1)'\bbM_l^{-1}\bbE_l'\bbQ)\E_l(\bbQ\bbE_l\bbM_l^{-1}(\bbb_l+\bbE_l'\bba_1))+1.\nonumber
\end{align}
In the proof of Theorem \ref{clt1}, we have shown that
\begin{align*}
 & p^{-4}\tr(\E_l\bbQ\bbE_l\bbM_l^{-1}\bbE_l'\bba_1\bba_1\bbE_l\bbM_l^{-1}\bbE_l'\bbQ)^2\\
&+p^{-2}2\E_l(\bba_1\bbE_l\bbM_l^{-1}\bbE_l'\bbQ)\E_l(\bbQ\bbE_l\bbM_l^{-1}\bbE_l'\bba_1)+1\\
=&\left(\frac{l-1}{n-k-l+1}+1\right)^2+o_p(1).
\end{align*}
By the same procedure and the assumptions in Theorem \ref{clt2}, we can also have that
\begin{align*} p^{-4}\tr(\E_l\bbQ\bbE_l\bbM_l^{-1}\bbb_l\bbb_l'\bbM_l^{-1}\bbE_l'\bbQ)^2=\left(\frac{\sum_{i=1}^{l-1}b_i^2}{n-k-l+1}\right)^2+o_p(1),
\end{align*}
\begin{align*} p^{-4}\tr(\E_l\bbQ\bbE_l\bbM_l^{-1}\bbb_l\bba_1 \bbE_l\bbM_l^{-1}\bbE_l'\bbQ)^2=\frac{(l-1)\sum_{i=1}^{l-1}b_i^2}{n-k-l+1}+o_p(1),
\end{align*}	
\begin{align*}
p^{-2}\E_l(\bbb_l\bbM_l^{-1}\bbE_l'\bbQ)\E_l(\bbQ\bbE_l\bbM_l^{-1}\bbE_l'\bba_1)=o_p(1),
\end{align*}
and
\begin{align*}
p^{-2}\E_l(\bbb_l \bbM_l^{-1}\bbE_l'\bbQ)\E_l(\bbQ\bbE_l\bbM_l^{-1}\bbb_l) =\frac{\sum_{i=1}^{l-1}b_i^2}{n-k-l+1}+o_p(1),
\end{align*}
which together with \eqref{trElGa21} and \eqref{p4Mlp} implies
\begin{align*}
\frac{1}{p} \sum_{l=1}^p \E_{l-1}\tr(\E_l\bmGa_l\E_l\bmGa_l)=\frac{1}{p}\sum_{l=1}^p\left(\frac{(l-1+\sum_{i=1}^{l-1}b_i^2)}{n-k-l+1}+1\right)^2+o_p(1).
\end{align*}

For $ \E_l\bmga_l\E_l\bmga_l'$, by the notation
$\widetilde\bbM_l=\frac1p\widetilde\bbE_l'\bbQ\widetilde\bbE_l$, we have that
\begin{align*}
 &\E_l\bmga_l'\E_l\bmga_l\\
 =&\E_l[(p^{-1} b_{l}\bbQ\bbE_l\bbM_l^{-1}(\bbb_l+\bbE_l'\bba_{1})-\bba_1b_{l})'((p^{-1} b_{l}\bbQ\bbE_l\bbM_l^{-1}(\bbb_l+\bbE_l'\bba_{1})-\bba_1b_{l})]\\
 =&\E_l[p^{-2}b_{l}^2(\bbb_l+\bba_{1}'\bbE_l)\bbM_l^{-1}\bbE_l'\bbQ\widetilde\bbE_l\widetilde\bbM_l^{-1}(\bbb_l+\widetilde\bbE_l'\bba_{1})+b_l^2].
\end{align*}
Applying the inversion formula of block matrix \eqref{beamel} again, we obtain that
\begin{align*}
p^{-1} \sum_{l=1}^p\E_{l-1}(\E_l\bmga_l'\E_l\bmga_l)=p^{-1}\sum_{l=1}^pb_l^2\left(\frac{(l-1+\sum_{i=1}^{l-1}b_i^2)}{n-k-l+1}+1\right)+o_p(1).
\end{align*}
Then by applying Lemma \ref{BPlmt}, we have that as $n\to\infty$,
\begin{align*}
 &\frac{1}{p}\sum_{l=1}^p\E_{l-1}(\tr(\E_l\bmGa_l\E_l\bmGa_l) +2\E_l\bmga_l\E_l\bmga_l')\\
 =&\frac{1}{p}\sum_{l=1}^p
 \left(\left(\frac{l-1+\sum_{i=1}^{l-1}b_i^2}{n-k-l+1}+1\right)^2+2b_l^2\left(\frac{l-1+\sum_{i=1}^{l-1}b_i^2}{n-k-l+1}+1\right)\right)+o_p(1)\\
 =&\frac{n}{p}\int_0^{c_n}\left[\left(\frac{t(1+\delta_1)}{1-\alpha_n-t}+1\right)^2+2\delta_1\left(\frac{t(1+\delta_1)}{1-\alpha_n-t}+1\right)\right]dt+o_p(1)\\
 =&\frac{(1-\alpha_n)(1+2\delta_1)+c_n\delta_1^2}{1-\alpha_n-c_n}+o_p(1).
\end{align*}

 We now turn to the term $\E_{l-1}\tr(\E_l\bmGa_l\circ\E_l\bmGa_l)$.
By the notation that $\bbu_j$ is an $n$-dimensional column vector with the j-th element being 1 and 0 otherwise and repeating the same argument in the proof of Theorem \ref{clt1}, we can obtain that
\begin{align*}
 \E_{l-1}\tr(\E_l\bmGa_l\circ\E_l\bmGa_l)=\sum_{j=1}^n(\E\bbu_j'\bmGa_l\bbu_j)^2+o_p(1).
\end{align*}
As $\{e_{ij}\}$ are i.i.d., thus from the assumptions of this theorem, we have that
 \begin{align*}
& \sum_{j=1}^n(\E\bbu_j'\bmGa_l\bbu_j)^2\\
=&\sum_{j=1}^n( n^{-1}p^{-1}\E(\bbb_l+\bbE_l'\bba_1)'\bbM_l^{-1}(\bbb_l+\bbE_l'\bba_1)\\
 &-2p^{-1}\bbu_j'\bba_1\E(\bbb_l+\bbE_l'\bba_1)'\bbM_l^{-1}\bbE_l'\bbQ\bbu_j+
 (\bbu_j'\bba_1)^2]^2\\
 =&o(1).
 \end{align*} Then we conclude that
\begin{align*}
\frac{1}{p}\sum_{l=1}^p\E_{l-1}\tr(\E_l\bmGa_l\circ\E_l\bmGa_l)=o_p(1).
\end{align*}

Next, we will prove that the non-random sequence
 \begin{align*}
 \cM_{4}=o(1).
\end{align*}
By the notation $\bbM^{-1}=(M^{ij})$ and $\{e_{ij}\}$ are i.i.d., we have that
\begin{align*}
 &\E(\bbb'+\bba_{1}'\bbE)\bbM^{-1}(\bbE'\bba_{1}+\bbb)=\sum_{i,j=1}^p\E(b_i+\bba_{1}'\bbe_i)M^{ij}(b_j+\bba_{1}'\bbe_j)\\
 =&\bbb'\bbb\E M^{11}+p\E\bbe_1'\bba_{1}\bba_{1}'\bbe_1M^{11}+2\E\bba_{1}'\bbe_1M^{11}\sum_{i=1}^pb_i\\
 &+\sum_{i\neq j}^pb_ib_j\E M^{12}+p(p-1)\E\bba_{1}'\bbe_1M^{12}\bbe_2'\bba_{1}+{2(p-1)}\E\bba_{1}'\bbe_1M^{12}\sum_{i=1}^pb_i.
 \end{align*}
From the inverse matrix formula, we know that
\begin{align*}
 M^{11}=\frac1{\beta_1}
	=\frac1{\beta^{tr}_1}-\frac{\xi_1}{(\beta^{tr}_1
)^2}+\frac{\xi_1^2}{\beta_1(\beta^{tr}_1)^2}.
\end{align*}
and
\begin{align*}
 M^{12}
 =\frac{\bbe_1'\bbQ\bbE_1\bbM_1^{-1}\bbu_1}{p\beta^{tr}_1}-\frac{\xi_1\bbe_1'\bbQ\bbE_1\bbM_1^{-1}\bbu_1}{p(\beta^{tr}_1
)^2}+\frac{\xi_1^2\bbe_1'\bbQ\bbE_1\bbM_1^{-1}\bbu_1}{p\beta_1(\beta^{tr}_1)^2}.
\end{align*}

Then it follows from \eqref{betr}, \eqref{Exil} and the H\"{o}lder's inequality that
\begin{align*}
 p^{-1}\bbb'\bbb\E M^{11}-\frac{c_n\delta_1}{1-c_n-\alpha_n}=o(p^{-1/2})
\end{align*}
and
\begin{align*}
 \E\bbe_1'\bba_{1}\bba_{1}'\bbe_1M^{11}-\frac{c_n}{1-c_n-\alpha_n}=o(p^{-1/2}).
\end{align*}
By the facts that
\begin{align*}
&|\E\bba_1'\bbe_1\xi_1|\leq \sqrt{\E|\xi_1|^2}=O(p^{-1/2})
\end{align*}
and
\begin{align*}
 |\sum_{i=1}^pb_i|=O(p^{1/2}),
\end{align*}
we can obtain that
\begin{align*}
 \E\bba_{1}'\bbe_1M^{11}\sum_{i=1}^pb_i=O(1).
\end{align*}
It follows from
\begin{align*}
 |p^{-1}\sum_{i\neq j}^pb_ib_j|=O(1),~~|\E{\xi_1\bbe_1'\bbQ\bbE_1\bbM_1^{-1}\bbu_1}|=O(1)
\end{align*}
and the H\"{o}lder's inequality, we have that
\begin{align*}
 \sum_{i\neq j}^pb_ib_j\E M^{12}=o(p^{-1/2}).
\end{align*}
Therefore, similar to the proof of Theorem \ref{clt1}, we conclude that
 \begin{align*}
 \cM_{2}=o(1),
\end{align*}
and we complete the proof of this theorem.

%

%
%

\subsection{Some useful lemmas}
\begin{lemma}[Theorem 35.12 of \cite{Billingsley95P}]\label{martingaleCLT}
Suppose that for each $n$, $Y_{n 1}, Y_{n 2}, \ldots, Y_{n r_{n}}$ is a real martingale
difference sequence with respect to the increasing $\sigma$ -field $\left\{\mathcal{F}_{n j}\right\}$ having second moments. If as $n \rightarrow \infty,$
for each $\varepsilon>0$,

(i) $\quad \sum_{j=1}^{r_{n}} \E\left(Y_{n j}^{2} | \mathcal{F}_{n, j-1}\right) \stackrel{p }{\rightarrow} \sigma^{2}$,
where $\sigma^{2}$ is a positive constant;

(ii)
$\quad \sum_{j=1}^{r_{n}} \E\left(Y_{n j}^{2} I_{\left(\left|Y_{n j}\right| \geq \varepsilon\right)}\right) \rightarrow 0$,

then we have that
\[
\sum_{j=1}^{r_{n}} Y_{n j} \stackrel{\mathcal{D}}{\rightarrow} N\left(0, \sigma^{2}\right).
\]
\end{lemma}

\begin{lemma}[\cite{Burkholder71D}]\label{burkholder}
	Let $\left\{Y_{k}\right\}$ be a martingale difference sequence with respect to the increasing $\sigma$-field $\left\{\mathcal{F}_{k}\right\}.$ Then, for $\ell>1$,
\[
{\E}\left|\sum X_{k}\right|^{\ell} \leq K_{\ell} \left(\E\left(\sum\E(\left|X_{k}\right|^{2}|\mathcal{F}_{k-1})\right)^{\ell / 2}+\sum\E\left|X_{k}\right|^{\ell}\right).
\]

\end{lemma}

\begin{lemma}[Lemma 2.7 of \cite{BaiS98N}]\label{BSinq}
	 For $\bbe=(e_{1}, \ldots,e_{n})'$ i.i.d. standardized entries, $\bbA$ a $n\times n$ matrix, we have, for any $\ell \geq 2$
\[
{\E}\left|\bbe' \bbA \bbe-{\tr} \bbA\right|^{\ell} \leq K_{\ell}\left(\left({\E}\left|e_{1}\right|^{4} {\tr} \bbA\bbA'\right)^{\ell / 2}+{\E}\left|e_{1}\right|^{2 \ell} {\tr} (\bbA\bbA')^{\ell / 2}\right).
\]

\end{lemma}
\begin{lemma}[Lemma 7.2 of \cite{BaiY05C}]\label{BYinq}
 Let $\mathbf{e}=\left(e_{1}, \ldots, e_{n}\right)'$ be a random $n$-vector with i.i.d. standardized entries. Suppose $E\left|e_{i}\right|^{4}<\infty$ and $\left|e_{i}\right| \leq \eta_{n} \sqrt{n}$ with $\eta_{n} \rightarrow 0$
slowly. Assume that $\bbA$ is a symmetric matrix of order $n$ bounded in norm by $M$. Then, for any given $2 \leqslant \ell\leqslant b \log \left(n \eta_{n}^{2}\right)$ with some $b>1,$ there exists a constant $K$ such that
$$
\mathbb{E}\left|\bbe' \bbA \bbe-\operatorname{tr}(\bbA)\right|^{\ell} \leq n^{\ell}\left(n \eta_{n}^{4}\right)^{-1}\left(M K \eta_{n}^{2}\right)^{\ell}.
$$
\end{lemma}
\begin{lemma}\label{quadratic_equ}
Let $\mathbf{B}$ and $\mathbf{C}$ be $n \times n$ matrices. Let $\bbd$ be a $n$-vector. Let $\mathbf{e}=\left(e_{1}, \ldots, e_{n}\right)'$ be a random $n$-vector with i.i.d. standardized entries. Let $\tau:=\E e_{i}^{4}-3$. Then, we have that
\[
{\E}\left\{\left(\mathbf{e}' \mathbf{B} \mathbf{e}-\operatorname{tr} \mathbf{B}\right)\left(\mathbf{e}' \mathbf{C} \mathbf{e}-\operatorname{tr} \mathbf{C}\right)\right\}=\operatorname{tr}(\mathbf{B} \mathbf{C})+\operatorname{tr}\left(\mathbf{B} \mathbf{C}'\right)+\tau \sum_{i=1}^{n} b_{i i} c_{i i}
\]
and
$${\E}(\mathbf{e}' \mathbf{B} \mathbf{e}\mathbf{e}'\bbd)=\E e_1^3 \sum_{i=1}^nb_{ii}d_i.
$$
\end{lemma}
\begin{lemma}[Lemma 3.1 of \cite{BaiP12L}]\label{BPlmt}
 Let $\left\{\mathbf{d}_{n}=\left(d_{1}, \ldots, d_{n}\right)'\right\}$ be a sequence of unit vectors with $\max _{k \leq n}\left|d_{k}\right| \rightarrow 0$.
There is a permutation $m$ of $\{1, \ldots, n\}$ given by
$$
\left(\begin{array}{cccc}
1 & 2 & \cdots & n \\
m(1) & m(2) & \cdots & m(n)
\end{array}\right)
$$
such that $\mathbf{d}_{f}=\left(d_{f(1)}, \ldots, d_{f(n)}\right)$ and $F_{n m}$ tends to a uniform distribution over the interval $(0,1),$ where $F_{n m}$ is a distribution function defined by
$$
F_{n m}(t)=\sum_{i \leq n t}\left|d_{f(i)}\right|^{2}.
$$
\end{lemma}


%

\end{document}